\documentclass[11pt]{amsart}
\usepackage[dvips,pdftex]{graphicx,color}
\usepackage{amssymb,amscd,latexsym,epsfig,xypic,hyperref}
\usepackage[mathscr]{euscript}
\allowdisplaybreaks
\DeclareFontFamily{OT1}{rsfs}{}
\DeclareFontShape{OT1}{rsfs}{n}{it}{<-> rsfs10}{}
\DeclareMathAlphabet{\curly}{OT1}{rsfs}{n}{it}

\makeatletter
\newcommand{\eqnum}{\refstepcounter{equation}\textup{\tagform@{\theequation}}}
\makeatother

\renewcommand\;{\hspace{.6pt}}

\newcommand\PP{\mathbb P}
\newcommand\LL{\mathbb L}
\newcommand\C{\mathbb C}
\newcommand\Q{\mathbb Q}
\newcommand\R{\mathbb R}
\newcommand\N{\mathbb N}
\newcommand\Z{\mathbb Z}

\newcommand\VW{\mathsf{VW}}

\newcommand\cO{\mathcal O}
\newcommand\Ohat{\widehat\cO^{\;\vir}_{\!M}}

\newcommand\cB{\mathcal B}

\newcommand\cE{\mathcal E}
\newcommand\E{\mathsf E}
\newcommand\cF{\mathcal F}

\newcommand\cI{\mathcal I}

\newcommand\cM{\mathcal M}
\newcommand\cN{\mathcal N}
\newcommand\cP{\mathcal P}

\newcommand\cZ{\mathcal Z}
\renewcommand\t{\mathfrak t}

\renewcommand\({\big(}
\renewcommand\){\big)}
\makeatletter
\newcommand{\so}{\ \ext@arrow 0359\Rightarrowfill@{}{\hspace{3mm}}\ }
\makeatother
\newcommand{\rt}[1]{\xrightarrow{\ #1\ }}
\newcommand\To{\longrightarrow}

\newcommand\into{\hookrightarrow}
\newcommand\INTO{\ \ar@{^(->}[r]<-.2ex>}
\newcommand{\Into}{\ensuremath{\lhook\joinrel\relbar\joinrel\rightarrow}}

\renewcommand\_{^{}_}
\newcommand\take{\backslash}

\newfont{\bigtimesfont}{cmsy10 scaled \magstep5}
\newcommand{\bigtimes}{\mathop{\lower0.9ex\hbox{\bigtimesfont\symbol2}}}
\renewcommand\={\ =\ }
\newcommand\dbar{\overline\partial}
\newcommand\udot{^{\bullet}}

\newcommand\Td{\operatorname{Td}}
\newcommand\ev{\operatorname{ev}}
\newcommand\Gr{\operatorname{Gr}}

\newcommand\ch{\operatorname{ch}}
\newcommand\rk{\operatorname{rank}}
\newcommand\vir{\operatorname{vir}}
\newcommand\red{\operatorname{red}}
\newcommand\vd{\operatorname{vd}}
\newcommand\tr{\operatorname{tr}}

\newcommand\id{\operatorname{id}}
\newcommand\vol{\operatorname{vol}}

\newcommand\Hom{\operatorname{Hom}}
\renewcommand\hom{\curly H\!om}
\newcommand\End{\operatorname{End}}
\newcommand\Ext{\operatorname{Ext}}
\newcommand\ext{\curly Ext}

\newcommand\Aut{\operatorname{Aut}}
\newcommand\Pic{\operatorname{Pic}}

\newcommand\Spec{\operatorname{Spec}\;}
\newcommand\Hilb{\operatorname{Hilb}}

\newcommand\Sym{\operatorname{Sym}}

\newcommand\beq[1]{\begin{equation}\label{#1}}
\newcommand\eeq{\end{equation}}
\newcommand\beqa{\begin{eqnarray*}}
\newcommand\eeqa{\end{eqnarray*}}

\newcommand\arXiv[1]{\href{http://arxiv.org/abs/#1}{arXiv:#1}}
\newcommand\mathAG[1]{\href{http://arxiv.org/abs/math/#1}{math.AG/#1}}
\newcommand\hepth[1]{\href{http://arxiv.org/abs/hep-th/#1}{hep-th/#1}}

\DeclareRobustCommand{\SkipTocEntry}[3]{}
\makeatletter
\newcommand\@dotsep{4.5}
\def\@tocline#1#2#3#4#5#6#7{\relax
  \ifnum #1>\c@tocdepth 
  \else
    \par \addpenalty\@secpenalty\addvspace{#2}%
    \begingroup \hyphenpenalty\@M
    \@ifempty{#4}{%
      \@tempdima\csname r@tocindent\number#1\endcsname\relax
    }{%
      \@tempdima#4\relax
    }%
    \parindent\z@ \leftskip#3\relax \advance\leftskip\@tempdima\relax
    \rightskip\@pnumwidth plus1em \parfillskip-\@pnumwidth
    #5\leavevmode #6\relax
    \leaders\hbox{$\m@th
      \mkern \@dotsep mu\hbox{.}\mkern \@dotsep mu$}\hfill
    \hbox to\@pnumwidth{\@tocpagenum{#7}}\par
    \nobreak
    \endgroup
  \fi}
\makeatother

\makeatletter \@addtoreset{equation}{section} \makeatother
\renewcommand{\theequation}{\thesection.\arabic{equation}}
\newtheorem{defn}[equation]{Definition}
\newtheorem{thm}[equation]{Theorem}
\newtheorem*{thm*}{Theorem}

\newtheorem{cor}[equation]{Corollary}
\newtheorem{prop}[equation]{Proposition}
\newtheorem{ass}[equation]{Assumption}

\newtheorem{conj}[equation]{Conjecture}
\newtheorem{rmk}[equation]{Remark}

\title{Equivariant $K$-theory and refined Vafa-Witten invariants}
\author{Richard P. Thomas}

\begin{document}
\maketitle
\begin{abstract} \noindent In \cite{MT2} the Vafa-Witten theory of complex projective surfaces is lifted to oriented $\mathbb C^*$-equivariant cohomology theories. Here we study the $K$-theoretic refinement. It gives rational functions in $t^{1/2}$ invariant under $t^{1/2}\leftrightarrow t^{-1/2}$ which specialise to numerical Vafa-Witten invariants at $t=1$.

On the ``instanton branch" the invariants give the virtual $\chi\_{-t}$-genus refinement of G\"ottsche-Kool, extended to allow for strictly semistable sheaves. Applying modularity to their calculations gives predictions for the contribution of the ``monopole branch". We calculate some cases and find perfect agreement. We also do calculations on K3 surfaces, finding Jacobi forms refining the usual modular forms, proving a conjecture of G\"ottsche-Kool.

We determine the $K$-theoretic virtual classes of degeneracy loci using Eagon-Northcott complexes, and show they calculate refined Vafa-Witten invariants. Using this Laarakker \cite{La1} proves universality results for the invariants.
\end{abstract}


\renewcommand\contentsname{\vspace{-1cm}}
\tableofcontents

\section{Introduction}
\noindent\textbf{Numerical.}
Vafa-Witten invariants should exist for all Riemannian 4-manifolds $S$ \cite{VW}, but mathematicians have yet to find a general definition. When $(S,\cO_S(1)$ is a smooth complex projective surface the invariants were defined in \cite{TT1, TT2}. However, physicists are by now less interested in \emph{numerical} Vafa-Witten invariants, which they mostly know how to calculate\footnote{Mathematicians should consider their calculations as conjectures we cannot yet prove.} in rank 2. They care more about the \emph{refined} Vafa-Witten invariants which arise in topologically twisted maximally supersymmetric 5d super Yang-Mills theory, but which do not have a mathematical definition.\medskip

\noindent\textbf{Joyce/Kontsevich-Soibelman.}
So we would like to refine the numerical invariants \cite{TT1,TT2} on a smooth complex polarised surface $(S,\cO_S(1))$. Those numerical invariants are closely related to local DT invariants of the local Calabi-Yau 3-fold $X=K_S$. (In fact when $H^1(\cO_S)=H^2(\cO_S)$ they are precisely local DT invariants, as studied in \cite{GSY1, GSY2} for instance.) They count certain compactly supported 2-dimensional torsion sheaves on $X$ via localisation with respect to the obvious $T=\C^*$ action on $X$.

If they were defined by Euler characteristic localisation --- weighted by the Behrend function \cite{Be} --- they would have an obvious refinement given by the work of Team Joyce and Kontsevich-Soibelman. But Euler characteristic localisation gives the wrong answer, and in fact the invariants of \cite{TT1,TT2} are defined by \emph{virtual localisation}.\medskip

\noindent\textbf{Nekrasov-Okounkov.}
For 3-folds $X$ with a $\C^*$ action, Nekrasov and Okounkov \cite{NO} give a different refinement of DT theory via equivariant virtual \emph{$K$-theoretic} invariants. This means replacing the length of the 0-dimensional virtual cycle (the classical DT invariant) by the holomorphic Euler characteristic of the virtual structure sheaf. This gives the same numerical answer, but also allows for a refinement by using the $T$ action to promote dimensions of cohomology groups to characters of $T$. The result is rational functions of $t$ which specialise at $t=1$ to the old numerical invariants.

In fact Nekrasov-Okounkov twist by a square root of the virtual canonical bundle of the moduli space before taking (equivariant) holomorphic Euler characteristic. This is motivated by physics, relating the $\dbar$ operator to the Dirac operator. From an algebro-geometric point of view, it makes the refinement more symmetric: it is a \emph{rational function} in $t^{1/2}$ which, by Serre duality, is invariant under $t^{1/2}\leftrightarrow t^{-1/2}$. The choice of square root is equivalent to the choice of orientation data in the Joyce/Kontsevich-Soibelman theory. Fortunately in our setting there is a canonical choice on the $T$-fixed locus: see Proposition \ref{sqprop} below.\medskip

\noindent\textbf{Vafa-Witten refinement.}
Under certain circumstances Davesh Maulik \cite{Ma} proved that the $K$-theoretic and Joyce/Kontsevich-Soibelman refinement of DT theory coincide. (The most general refinement is based on the two variable Hodge-Deligne polynomial; here we are concerned with the Hirzebruch $\chi\_{-t}$ one variable specialisation.) So he suggested that it is natural to try to also refine Vafa-Witten theory using $T$-equivariant $K$-theory. That is what we do in this paper. We use the Vafa-Witten perfect obstruction theory of \cite{TT1,TT2} to produce a virtual structure sheaf, and then twist by the square root \eqref{sqroot} of the virtual canonical bundle. We then  use a virtual localisation formula to take equivariant holomorphic Euler characteristic. In Theorem \ref{VW=J} we reprove Maulik's result that this refinement recovers the refined DT invariant when $\deg K_S<0$.
\medskip

\noindent\textbf{Further refinements.}
In fact this is a special case of more general refinements. Recall \cite{BF} that virtual cycles come from intersecting a cone in a vector bundle $C\subset E$ over moduli space $M$ with the zero section $\iota:0_E\into E$ of the vector bundle. One can intersect these two cycles $[C]$ and $[0_E]$ in \emph{any} oriented cohomology theory. Traditionally we use Fulton-MacPherson intersection theory to get the virtual cycle in homology \cite{BF},
$$
[M]^{\vir}\=\iota^![C]\ \in\ H_*(M).
$$
In $K$-theory we instead take the (derived) tensor product of the structure sheaves of the two cycles \cite{FG, CFK},
$$
\cO_M^{\vir}\=[\cO_{0_E}\stackrel L\otimes\cO_C]\=[L\iota^*\cO_C]\ \in\ K_0(M).
$$
The result is slightly different --- differing by a Todd class, by virtual Riemann-Roch --- and therefore interesting! (Especially when we work equivariantly with respect to the $T$ action.)

To get the Nekrasov-Okounkov-twisted version of this used in our paper we instead take the intersection in $KO$-theory. This replaces the $K$-theoretic ``fundamental classes" $\cO_Z$ of submanifolds by their twists by $K_Z^{1/2}$ (this is the Atiyah-Bott-Shapiro complex orientation, and is well defined over $\Z$ only for spin manifolds).

The universal case is complex cobordism theory; see \cite{Sh, GK2} for instance. From this one can pass to all other oriented cohomology theories, such as ``topological modular forms".\footnote{\emph{tmf} is a deep theory over the integers \cite{Ho}; we only use its rational version which is trivial to define using the Witten genus and the Landweber exact functor theorem.} In Vafa-Witten theory, the three $T$-equivariant cohomology theories
$$
\big(\;\text{homology, $KO$-theory, tmf}\;\big)
$$
give rise to virtual versions of
$$
\Big(\textstyle{\int_M1},\ \widehat{A}(M),\text{ Witten genus}\;(M)\Big)
$$
of the Vafa-Witten moduli space $M$ respectively. On the ``instanton locus" these produce the
$$
\big(\;\text{Euler characteristic, Hirzebruch }\chi\_y\text{-genus, elliptic genus}\;\big)
$$
of the moduli space of instantons (or Gieseker stable sheaves) on the surface $S$. (This apparent paradox is because the Vafa-Witten obstruction theory on the instanton moduli space differs from its usual obstruction theory.)

Calculations give generating series which seem to be
$$
\text{(modular forms, Jacobi forms, Borcherds lifts of Jacobi forms)}
$$
respectively; in particular see \cite{GK2} for the instanton locus contributions in rank 2.

These refinements of Vafa-Witten theory are defined and studied in the forthcoming paper \cite{MT2}. In this paper we specialise the general definition from \cite{MT2} to $T$-equivariant $K$-theory and explore it in more detail.\medskip

\noindent\textbf{Results.}
In Section \ref{Ksection} we give a careful treatment of virtual $K$-theoretic localisation for $T$-equivariant $K$-theory on quasi-projective $T$-schemes with a $T$-equivariant perfect obstruction theory and \emph{compact $T$-fixed locus}. This allows us to define $K$-theoretic invariants for such schemes endowed with a choice of square root of the virtual canonical bundle. For simplicity, in this Introduction we state our results for \emph{symmetric} perfect obstruction theories; then Proposition \ref{sqprop} gives a \emph{canonical} choice \eqref{sqroot} of this square root.

\begin{thm*} Let $M$ be a quasi-projective $T$-scheme with compact $T$-fixed locus, and a $T$-equivariant symmetric perfect obstruction theory. Then the refined invariant
$$
\chi\_t\(M,\Ohat\)\ :=\ \chi\_t\!\left(\!M^T\!,\,\frac{\cO^{\vir}_{M^T}}{\Lambda\udot(N^{\vir})^\vee}\otimes K^{\frac12}_{M,\vir}\Big|_{M^T}\!\right)
$$
of Definition \ref{defdef} is a rational function of $t^{\frac12}$, invariant under $t^{\frac12}\leftrightarrow t^{-\frac12}$. It is deformation invariant and has poles only at roots of unity and the origin, but not at $t=1$. Specialising to $t=1$ recovers the numerical invariant defined by $T$-equivariant localisation,
$$
\left.\chi\_t\(M,\Ohat\)\right|_{t=1}\=\int_{[M^T]^{\vir}}\frac1{e(N^{\vir})}\,.
$$
\end{thm*}

\noindent\textbf{Stable case.}
In Section \ref{VWsec} we apply this to the Vafa-Witten moduli space for  a projective surface $S$ and a charge $\alpha\in H^*(S)$ for which semistability implies stability.\footnote{Here $\alpha$ is the total Chern class of the sheaves on $S$ underlying the Vafa-Witten Higgs pairs in our moduli space.}
The moduli space carries a symmetric perfect obstruction theory and a $T$ action inherited from the $T$ action on $X$.
The result is invariants which specialise to the numerical Vafa-Witten invariants \cite{TT1}
\beq{vwvw}
\VW_{\alpha}(t)\ \in\ \Q(t^{1/2}) \quad\text{such that}\quad
\VW_{\alpha}(1)\=\VW_{\alpha}\ \in\ \Q.
\eeq
This $\VW_\alpha(t)$ is made up of contributions from the two types of component of the $T$-fixed locus:
\begin{itemize}
\item the ``\emph{instanton branch}" of sheaves on $S$ pushed forward to $X$,
\item the ``\emph{monopole branch}" of $T$-equivariant sheaves supported on a nontrivial scheme theoretic thickening of $S\subset X$.
\end{itemize}
\medskip

\noindent\textbf{Semistable case.}
We tackle the general case in Section \ref{sscase}. We use Joyce-Song pairs to rigidify semistable sheaves as in \cite{TT2}. The resulting refined pair invariants $P^\perp_{\alpha,n}(t)$ are functions of the twisting parameter $n\gg0$ of the Joyce-Song pairs. According to Conjecture \ref{pechconj} they should be expressable in terms of certain universal functions in $n$,
$$
P^\perp_{\alpha}(n,t)\ =\ \mathop{\sum_{\ell\ge 1,\,(\alpha_i=\delta_i\alpha)_{i=1}^\ell:}}_{\delta_i>0,\ \sum_{i=1}^\ell\delta_i=1}
\frac{(-1)^\ell}{\ell!}\prod_{i=1}^\ell(-1)^{\chi(\alpha_i(n))} \big[\chi(\alpha_i(n))\big]_t\VW_{\alpha_i}(t) \hspace{-2mm}
$$
if $H^{0,1}(S)=0=H^{0,2}(S)$; otherwise we take only the first term in the sum:
$$
P^\perp_\alpha(n,t)\ =\ (-1)^{\chi(\alpha(n))-1}\big[\chi(\alpha(n))\big]\_t\VW_\alpha(t).
$$
The coefficients $\VW_\alpha(t)$ then define the refined Vafa-Witten invariants. Here $[\chi]\_t$ is the quantum integer
\beq{qq}
[\chi]\_t\ :=\ \frac{t^{\chi/2}-t^{-\chi/2}}{t^{1/2}-t^{-1/2}}\=
t^{-c}+t^{-c+1}+\cdots+t^{c-1}+t^c, \quad c:=\frac{\chi-1}2\,.
\eeq
Since $[\chi]\_t\to\chi$ as $t\to1$, Conjecture \ref{pechconj} specialises to Conjecture 6.5 of \cite{TT2}, now proved in many cases \cite{TT2, MT1}. Here we prove the refined conjecture in some situations.

\begin{thm*} Conjecture \ref{pechconj} holds, thus defining refined Vafa-Witten invariants $\VW_\alpha(t)$, in the following cases.
\begin{itemize}
\item When all semistable sheaves of charge $\alpha$ are stable. In this case $\VW_\alpha(t)$ recovers the invariants \eqref{vwvw}.
\item $\mathbf{K_S<0}$. When $\deg K_S<0$, for any charge $\alpha$. Here we recover refined\,\footnote{There is a refinement of DT theory based on the two-variable Hodge-Deligne polynomial $E(u,v)$. Here we use the one-variable specialisation based on the Hirzebruch $\chi\_{-t}:=E(t,1)$ genus; see Section \ref{K<0} for details and definitions.} DT invariants: $\VW_\alpha(t)=\mathsf{J}_\alpha(t)$.
\item $\mathbf{K_S=0}$. 
When $S$ is a K3 surface and $\alpha$ is either a primitive class or a prime multiple of a primitive class.
\item $\mathbf{K_S>0}$. When $p_g(S)>0$, for any charge $\alpha$ with prime rank, Laarakker \cite{La2} shows that the conjecture holds for the contribution of the monopole locus. He uses the vanishing Theorem \ref{cosecinj} to remove many components, and \cite{GT1,GT2} to calculate with the rest.
\end{itemize}
\end{thm*}

Recently Liu has proved the general case of Conjecture \ref{pechconj} \cite{Liu}.\medskip

\noindent\textbf{K3 surfaces.}
We are able to do extensive calculations when $S$ is a K3 surface.
The well-known $1/d^2$ multiple cover formula of DT theory is replaced by a $1/[d]_t^2$ multiple cover formula \eqref{refmc} in the refined setting. (This is a surprising contrast to the $1/d[d]_t$ refined multiple cover formula seen in DT theory --- see \cite[Section 6.7]{DM} for instance.) At the level of generating series we are led to the conjecture
\beq{conge}
\sum_k\VW_{r,k}(t)\;q^n\=\sum_{d|r}\frac1{[d]_t^2}\frac dr\,q^r\!\sum_{j=0}^{r/d-1}\widetilde\Delta\(e^{2dj\pi i/r}q^{\frac{d^2}r},t^d\)^{-1},
\eeq
where $\widetilde\Delta$ is the Jacobi form
$$
\widetilde\Delta(q,t)\ :=\ q\prod_{k=1}^\infty(1-q^k)^{20}(1-tq^k)^2(1-t^{-1}q^k)^2.
$$
This has now been proved in \cite{KK3}.

\begin{thm*} For $S$ a K3 surface \eqref{conge} holds.
\end{thm*}

\noindent\textbf{Modularity.}
On the instanton branch $\cM$ our refined Vafa-Witten invariants recover the virtual $\chi\_{-t}$-genus refinement studied by G\"ottsche-Kool on surfaces with $K_S>0$ \cite{GK1}. This is most easily seen when $\cM$ is smooth and unobstructed as a moduli space of fixed-determinant sheaves on $S$. Then its Vafa-Witten obstruction bundle is $\Omega_{\cM}\otimes\t$ so that
$$
\cO^{\vir}_{\!\cM}\=\Lambda\udot\;\mathrm{Ob}^*\ \cong\ \Lambda\udot(T_{\cM}\t^{-1}) \quad\text{and}\quad K_{\vir}\=K_{\cM}^2\t^{\;\dim\cM}.
$$
Therefore
$$
\widehat\cO^{\;\vir}_{\!\cM}\=\cO^{\vir}_{\!\cM}\otimes K_{\vir}^{\frac12}\=(-1)^{\dim\cM}\;\t^{-\dim\cM/2}\Lambda\udot(\Omega_{\cM}\t)
$$
so that $\chi\_t(\widehat\cO^{\;\vir}_{\!\cM})=(-1)^{\dim\cM}\;\t^{-\dim\cM/2}\chi\_{-t}(\cM)$.

Applying modularity to G\"ottsche-Kool's calculations of these invariants gives predictions for the contribution of the ``monopole branch". We calculate a small number of cases (which nonetheless take 9 pages of calculation) and find perfect, honest\footnote{I did the calculations repeatedly until they converged; only then did I allow Martijn Kool to tell me his prediction from \cite{GK3}; fortunately the results matched.} agreement.\medskip

\noindent\textbf{Nested Hilbert schemes.}
There are components of the monopole branch which are nested Hilbert schemes of $S$. In \cite{GT1, GT2} it was shown how to view these as degeneracy loci in smooth products of Hilbert schemes of $S$. This induces a virtual cycle which agrees with the one from Vafa-Witten theory. Its pushforward is described by the Thom-Porteous formula. This gives a more systematic way to compute numerical Vafa-Witten invariants as integrals over products of smooth Hilbert schemes.

In Section \ref{degensec} we describe $K$-theoretic analogues of these results, replacing Chern classes by Koszul resolutions and the Thom-Porteous formula by Eagon-Northcott complexes. The most straightforward result, relevant to nested Hilbert schemes of \emph{points} on a surface, is the following.

\begin{thm*}
Given a map of vector bundles $\sigma\colon E_0\to E_1$ over a smooth scheme $X$, the locus $Z$ where $\sigma$ is not injective carries a natural virtual structure sheaf whose pushforward to $X$ has $K$-theory class
\beq{E-N}
\iota_*\;\cO_Z^{\vir}\=\cO_X-\det(E_0-E_1)\otimes\bigwedge\!\!^r(E_1-E_0),
\eeq
where $r=\rk E_1-\rk E_0$.
\end{thm*}

When the degeneracy locus $Z$ has the correct codimension, the Eagon-Northcott complex of $\sigma\colon E_0\to E_1$ --- whose $K$-theory class is the right hand side of \eqref{E-N} --- is well known to resolve $\iota_*\;\cO_Z$. So the above result shows that even when it has the wrong codimension, its K-class is $\iota_*\;\cO^{\vir}_Z\in K_0(X)$.

Carlsson-Okounkov \cite{CO} express the Thom-Porteous class of \cite{GT1} in terms of Grojnowski-Nakajima operators. There are $K$-theoretic analogues of this in \cite{CNO, MO, SV} to which we intend to return.

We also relate $\cO_Z^{\vir}$ to the Vafa-Witten virtual structure sheaf when $Z$ is a nested Hilbert scheme. The upshot is that monopole branch contributions to \emph{refined} Vafa-Witten invariants can be computed from calculations on smooth products of Hilbert schemes of $S$. \medskip

\noindent\textbf{Laarakker.}
In \cite{La1} (stable case) and \cite{La2} (general case) Laarakker uses this to great effect on surfaces with $p_g>0$ (and $h^{0,1}=0$ for now). Things work best in prime rank, using the vanishing Theorem \ref{cosecinj} to eliminate many components of the monopole locus. The rest can be calculated via universal integrals over Hilbert schemes of points and curves on surfaces using the results of \cite{GT1,GT2}.

Moreover, the contributions from points and curves split, in an appropriate sense. The curves contribute Seiberg-Witten invariants (certain well-understood integrals over linear systems). Laarakker evaluates the contributions of Hilbert schemes of points via the method of \cite{EGL}. The result depends only on the curve class $\beta\in H_2(S,\Z)$ and the cobordism class of the surface --- and thus only on $c_1(S)^2,\,c_2(S)$ and $\beta^2$. Therefore these contributions can be calculated on K3 surfaces and toric surfaces (despite these not having $p_g>0$!).

\begin{thm*}\cite{La1}
Let $S$ be a minimal general type surface with $p_g(S)>0$ and $H_1(S,\Z)=0$ such that $K_S$ does not admit a square root. We work with rank 2 Higgs pairs $(E,\phi)$ with $\det E=K_S$. Then the monopole branch contributions to the refined Vafa-Witten generating series
$\sum_n\mathsf{VW}_{2,K_S,n}(t)\;q^n$ can be written
$$
A(t,q)^{\chi(\cO_S)}B(t,q)^{c_1(S)^2},
$$
where
$$
A(t,q),\ B(t,q)\ \in\ \Q(t^{1/2})(\!(q)\!),
$$
are universal functions, independent of $S$. 
\end{thm*}
Furthermore K3 calculations determine $A(t,q)$ completely, while by modularity and the work of G\"ottsche-Kool he knows what $B(t,q)$ \emph{should} be, and he can check this in low degree by toric computations.\medskip

\smallskip\noindent\textbf{Acknowledgements.} This paper originally grew out of a suggestion of Davesh Maulik, and is part of the joint work \cite{MT2}. Ideas, suggestions and calculations of Martijn Kool played a crucial role in the formulation of both \cite{MT2, TT1} and this paper, as did the insistence of Emanuel Diaconescu, Greg Moore and Edward Witten that there should be a refinement of the Vafa-Witten invariants defined in \cite{TT1, TT2}. Thanks to Ties Laarakker for suggesting the vanishing of Sections \ref{sectoc} and \ref{kthree}, and for correcting an error in \cite{TT2} which had propagated into this paper. I am grateful to Noah Arbesfeld, Pierrick Bousseau and Andrei Negut for sharing their expertise in $K$-theory and the Nekrasov-Okounkov refinement, and to Dominic Joyce and Sven Meinhardt for their help with refined DT theory. The author is partially supported by EPSRC grant EP/R013349/1. 

\section{$K$-theoretic virtual cycles}\label{Ksection}
The foundations of cohomological virtual cycles are laid down in \cite{BF, LT}; we use the notation from \cite{BF}. The foundations for $K$-theoretic virtual cycles (or ``\emph{virtual structure sheaves}") are laid down in \cite{CFK, FG}; we use the notation from \cite{FG}.

\subsection{Virtual cycle and virtual structure sheaf}
Let $M$ be a quasi-projective scheme with a perfect obstruction theory $E\udot\to\LL_M$ supported in degrees $[-1,0]$. That is, $E\udot$ is a 2-term complex $E^{-1}\to E^0$ of vector bundles on $M$ such that the map $E\udot\to\LL_M$ induces an isomorphism on $h^0$ and a surjection on $h^{-1}$. We call $\LL_M^{\vir}:=E\udot$ the virtual cotangent bundle of $M$, of rank $\vd:=\rk E^0-\rk E^{-1}$ and determinant
$$
K_{\vir}\ :=\ \det E\udot\=\det E^0\otimes\(\!\det E^1\)^{-1}.
$$
Dualising, we set $E_i:=(E^{-i})^*$ to get the virtual tangent bundle
$$
T^{\vir}_M\ :=\ E_\bullet\=(E\udot)^\vee\=\(\LL_M^{\vir}\)^\vee.
$$

By \cite{BF} this data defines a cone $C\subset E_1$ from which we may define $M$'s \emph{virtual cycle}
$$
\big[M\big]^{\vir}\ :=\ \iota_0^![C]\ \in\ A_{\vd}(M)
$$
and its \emph{virtual structure sheaf} \cite{FG}
\beq{vss}
\cO_M^{\vir}\ :=\ \big[L\iota_0^*\,\cO_C\big]\ \in\ K_0(M),
\eeq
where $\iota\_0\colon M\to E_1$ is the zero section. (Since $\iota\_0$ is a regular embedding, $L\iota_0^*\,\cO_C$ is a bounded complex.) If $M$ is compact the virtual Riemann-Roch theorem of \cite[Corollary 3.4]{FG} then gives
\beq{chivir}
\chi(\cO^{\vir}_M)\=\int_{[M]^{\vir}}\Td\(T^{\vir}_M\).
\eeq
In particular, if $M$ also has virtual dimension zero $\vd=0$ we can use either the virtual structure sheaf or the virtual cycle to define the same numerical invariant
\beq{num}
\chi(\cO^{\vir}_M)\=\int_{[M]^{\vir}}1\ \in\ \Z \qquad\mathrm{when}\ \vd=0.
\eeq

\subsection{Twisted virtual structure sheaf} Via the medium of Nekrasov and Okounkov, physics teaches us that we should choose a square root of the virtual canonical bundle and work instead with the \emph{twisted/modified/sym\-metrised}\footnote{For instance, this symmetrisation will lead, via Serre duality, to the $t^{1/2}\leftrightarrow t^{-1/2}$ symmetry of Proposition \ref{tt-1}.} virtual structure sheaf, 
\beq{twvss}
\Ohat\ :=\ K_{\vir}^{\frac12}\otimes\cO^{\vir}_{M}.
\eeq 
This paper is mainly concerned with the virtual $K$-theoretic invariant $\chi(\Ohat)$. When $M$ is compact the virtual Riemann-Roch theorem \cite[Corollary 3.4]{FG} gives the following cohomological expression for it,
\beq{twRR}
\chi(\Ohat)\=\int_{[M]^{\vir}}\ch\!\Big(\!K_{\vir}^{\frac12}\Big)\Td\(T^{\vir}_M\),
\eeq
modifying \eqref{chivir}. Of course in virtual dimension zero this makes no difference and we recover \eqref{num},
$$
\chi(\Ohat)\=\chi(\cO^{\vir}_M)\=\int_{[M]^{\vir}}1 \qquad\mathrm{when}\ \vd=0,
$$
but it will make a big difference to its refinement when we work \emph{equivariantly}. This will also allow us to fix the ambiguity in the choice of $K_{\vir}^{1/2}$, because on the fixed locus there is a \emph{canonical} choice.

\begin{prop} \label{sqprop}
Let $M$ be a quasi-projective scheme with a $T=\C^*$ action with projective fixed locus $M^T$. Suppose $M$ has a $T$-equivariant symmetric perfect obstruction theory $E\udot\to\LL_M$. Then $K_{M,\vir}\big|_{M^T}$ admits a canonical square root\;\footnote{We are abusing the notation $\ \cdot\ |_{M^T}$ since we have not shown this line bundle extends to $M$. The point is we will only need it on $M^T$. And our square root is only equivariant for the action of the \emph{double cover of $T$}, since we need to use $\t^{1/2}$.}
\beq{sqroot}
K^{\frac12}_{M,\vir}\Big|_{M^T}\ :=\ 
\det\!\Big(E\udot|\_{M^T}\Big)^{\ge0}\,\t^{\frac12r_{\mbox{\tiny{${}_{\ge0}$}}}}
\eeq
Here $\(E\udot|\_{M^T}\)^{\ge0}$ denotes the part of $E\udot|_{M^T}$ with nonnegative $T$-weights, and $r\_{\ge0}$ is its rank.
\end{prop}

\begin{proof}
Decompose the virtual cotangent bundle in weight spaces for the $T$ action,
$$
E\udot\big|_{M^T}\=\bigoplus_{i\in\Z}F^i\t^i.
$$
Here each $F^i$ is a $T$-fixed two-term complex of vector bundles on $M^T$ of (super)rank $r_i:=\rk(F^i)$. The symmetry of the obstruction theory,
$$
(E\udot)^\vee\ \cong\ E\udot[-1]\otimes\t
$$
implies that
$$
(F^i)^\vee\ \cong\ F^{-i-1}[-1].
$$
Therefore
\begin{align*}
K_{M,\vir}\=\det E\udot&\=\bigotimes_{i<0}\det(F^i\t^i)\,\otimes\ \bigotimes_{i\ge0}\det(F^i\t^i) \\
&\=\bigotimes_{i<0}\det\((F^i\t^i)^\vee[1]\)\,\otimes\ \bigotimes_{i\ge0}\det(F^i\t^i) \\
&\=\bigotimes_{i<0}\det(F^{-i-1}\t^{-i})\,\otimes\ \bigotimes_{i\ge0}\det(F^i\t^i) \\
&\=\bigotimes_{i\ge0}\det(F^i\t^{i+1})\,\otimes\ \bigotimes_{i\ge0}\det(F^i\t^i) \\
&\=\bigg(\bigotimes_{i\ge0}\det(F^i\t^i)\bigg)^{\!\otimes2}\t^{r_0+r_1+\cdots} \\
&\=\Big(\!\det\(E\udot|\_{M^T}\)^{\ge0}\Big)^{\otimes2}\t^{r_{\ge0}}.
\qedhere \end{align*}
\end{proof}

\subsection{Localisation}\label{sec*loc} Suppose as above $T:=\C^*$ acts on both $M$ and its perfect obstruction theory $E\udot\to\LL_M$. Then on its fixed locus $\iota\colon M^T\into M$ we get a splitting
$$
\iota^*E\udot\=(E\udot)^T\ \oplus\ \(N^{\vir}\)^*
$$
into fixed and moving parts (i.e. weights 0 and nonzero). By \cite{GP} the fixed part $(E\udot)^T$ defines a perfect obstruction theory on $M^T$ (and so a virtual structure sheaf \eqref{vss}). We call the dual of the moving summand the \emph{virtual normal bundle} $N^{\vir}$.
The virtual localisation formula of \cite{GP} states that
\beq{GP}
\iota_*\left(\frac1{e(N^{\vir})}\cap\big[M^T\big]^{\vir}\right)\=[M]^{\vir}\ \in\ A_{\vd}^T(M)\otimes\_{\Z[t]}\Q[t,t^{-1}].
\eeq
Here we consider the $T$-equivariant Chow homology to be a module over $H^*(BT)=\Z[t]$, which we localise, inverting the equivariant parameter $t$ --- the first Chern class of the weight 1 irreducible representation $\t$ of $T$. This ensures that, writing $N^{\vir}$ as a two-term complex of $T$-equivariant vector bundles $N_0\to N_1$ \emph{all of whose weights are nonzero}, the $c_{\mathrm{top}}(N_i)$ are invertible. Thus we may define $e(N^{\vir}):=c_{\mathrm{top}}(N_0)/c_{\mathrm{top}}(N_1)$.

To mimic this in $K$-theory we use the module structure of $K^0_T(M)$ over $K^0_T(\mathrm{point})=\Z[t,t^{-1}]$ to localise to the field of fractions\footnote{In fact it would be sufficient to invert $t^i-1$ for all $i=1,2,3,\ldots\,$.} $\Q(t)$. Adjoining $t^{1/2}$ (so we can lift our choice of square root \eqref{sqroot} to localised $T$-equivariant $K$-theory), we work in
$$
K^0_T(M)\otimes\_{\Z[t,t^{-1}]}\Q\(t^{\frac12}\).
$$

Now applying \eqref{GP} to \eqref{twRR}, and using the notation $\Lambda\udot E:=\sum_{i\ge0}(-1)^i\Lambda^iE$ in $K$-theory, we find
\begin{eqnarray}
\chi(\Ohat) &=& \int_{\big[M^T\big]^{\vir}}\frac{\ch\!\Big(\!K_{\vir}^{\frac12}\big|_{M^T}\Big)\Td\(T^{\vir}_{M^T}\)\Td(N^{\vir})}{e(N^{\vir})} \nonumber \\
&=& \int_{\big[M^T\big]^{\vir}}\frac{\ch\!\Big(\!K_{\vir}^{\frac12}\big|_{M^T}\Big)}{\ch(\Lambda\udot(N^{\vir})^\vee)}\Td\(T^{\vir}_{M^T}\) \nonumber \\
&=& \chi\!\left(\frac{\cO^{\vir}_{\!M^T}\otimes K^{\frac12}_{\vir}\big|_{M^T}}{\Lambda\udot(N^{\vir})^\vee}\right), \label{locchi}
\end{eqnarray}
the last line from the virtual Riemann-Roch theorem \cite[Corollary 3.4]{FG} on $M^T$. This suggests there should be a $K$-theoretic localisation formula
\beq{Kvirloc}
\iota_*\;\frac{\cO^{\vir}_{M^T}}{\Lambda\udot(N^{\vir})^\vee}\=\cO^{\vir}_{M}\ \in\ K^0_T(M)\otimes\_{\Z[t,t^{-1}]}\Q\(t^{\frac12}\),
\eeq
from which \eqref{locchi} would follow by taking $\chi\(\ \cdot\ \otimes K^{\frac12}_{\vir}\)$.

Such a result is proved in \cite[Theorem 5.3.1]{CFK} for $(M,E\udot)$ which can be enhanced to a $[0,1]$-dg-manifold structure. More recently Qu has proved \eqref{Kvirloc} for \emph{any} $T$-equivariant $(M,E\udot)$ \cite{Qu}.

For the first version of this paper I was unaware of Qu's work, and did not want to have to prove that Vafa-Witten moduli spaces $M$ admit a $T$-equivariant $[0,1]$-dg-manifold structure (though they certainly do). So I proved a weaker statement --- a $T$-equivariant version of \eqref{locchi}, which is sufficient for our purposes. I have kept that proof  --- which is Proposition \ref{abc} below --- since it demonstrates the compatiblity of $K$-theoretic and cohomological localisation under virtual Riemann-Roch. So we turn to this now.

\subsection{Refinement}
We are interested in situations where $M$ may be noncompact, but carries a $T$ action with \emph{compact fixed locus} $M^T$. Then \eqref{chivir} and \eqref{twRR} make no obvious sense, but their natural refinements --- given by replacing (super)ranks of graded vector spaces by characters of (virtual) $T$-modules --- are perfectly well-defined. That is, we define
\beq{character}
\chi\_t\!\left(\bigoplus_i\mathfrak t^{a_i}-\bigoplus_j\mathfrak t^{b_j}\right)\ :=\ \sum_it^{a_i}-\sum_jt^{a_j}.
\eeq
Here the left hand side may involve infinite direct sums; then the right hand side will be a sum of power series in $t^{1/2}$ and $t^{-1/2}$ which, in our situation, will be expansions of rational functions in $\Q(t^{1/2})$. When the sums are finite we may set $t=1$ and recover the usual Euler characteristic or (super)rank.

Applying \eqref{character} to $R\Gamma\(M,\Ohat\)$ gives our refinement of the integer \eqref{twRR}.
\beq{chitdef}
\chi\_t\(M,\Ohat\)\ :=\ \chi\_t\(R\Gamma\(M,\Ohat\)\)\ \in\ \Q\(t^{\frac12}\).
\eeq
When $M$ is compact this specialises to \eqref{twRR} at $t=1$, but \eqref{chitdef} makes sense more generally.

\begin{prop}\label{abc} Suppose $M^T$ is compact and we have chosen a square root $K^{1/2}_{\vir}\in K^T_0(M)(t^{1/2})$. Then the refined invariant \eqref{chitdef} can be calculated on $M^T$ by localisation as
\beq{chiloc}
\chi\_t\(M,\Ohat\)\=\chi\_t\!\left(\!M^T\!,\,\frac{\cO^{\vir}_{M^T}}{\Lambda\udot(N^{\vir})^\vee}\otimes K^{\frac12}_{M,\vir}\Big|_{M^T}\!\right)\,\in\ \Q\(t^{\frac12}\).
\eeq
\end{prop}

\begin{proof}
This follows directly from the localisation formula \eqref{Kvirloc} of \cite{CFK} when $M$ is a $[0,1]$-dg-manifold acted on by $T$, and from \cite{Qu} more generally. Alternatively, we can repeat the argument of \eqref{locchi} in an equivariant setting, replacing the virtual Riemann-Roch formula of \cite[Corollary 3.4]{FG} by its equivariant analogue.

We use the standard trick of approximating the homotopy quotient $M\times_T ET$ over $BT=\PP^\infty$ by the $M$-bundle
$$
p\_N\colon\ M\times_T(\C^{N+1}\take\{0\})\To\PP^N.
$$
(Here $T$ acts with weight 1 on $\C^{N+1}$.) Applying the virtual Grothendieck-Riemann-Roch theorem of \cite[Theorem 3.3]{FG} to $p\_N$ gives
\beq{1}
\ch\!\(Rp\_{N*\,}\Ohat\)\=p\_{N*}\!\left(\!\ch\!\Big(\!K^{\frac12}_{\vir}\!\Big)\Td\(T^{\vir}_M\)\cap[M]^{\vir}\right).
\eeq
As $N$ increases this is a sequence of compatible cohomology classes over $\PP^N\subset\PP^{N+1}\subset\cdots$, defining a class in
$$
\varprojlim H^*\(M\times_T(\C^{N+1}\take\{0\}),\R\)\=
H^*_T(M,\R)\otimes\_{\R[t]}\R[\![t]\!].
$$
By \eqref{GP} its class in $H^*_T(M,\R)\otimes_{\R[t]}\R(\!(t)\!)$ is equal to
\beq{1.5}
\ch\!\(Rp\_{N*\,}\Ohat\)\=
p\_{N*}\bigg(\,\frac{\ch\!\Big(\!K^{\frac12}_{\vir}\!\Big)\!\Td\!\(T^{\vir}_M\)}{e(N^{\vir})}\cap\big[M^T\big]^{\vir}\bigg).
\eeq
Again, we spell out what this means in terms of the finite dimensional models. We can expand $1/e(N^{\vir})$ as an $H^*(M)$-valued Laurent series in $t^{-1}$. Write it as a sum of terms $a_it^i$, where $a_i\in H^{-2\rk(N^{\vir})-2i}(M)$. This therefore vanishes for $-\rk(N^{\vir})-i>\dim M$, i.e. for $i$ sufficiently negative. Therefore $1/e(N^{\vir})$ is a Laurent polynomial in $t$, and we may pick $n\gg0$ such that $t^n\cdot1/e(N^{\vir})\in H^*_T(M,\R)\otimes_{\R[t]}\R[\![t]\!]\subset H^*_T(M,\R)\otimes_{\R[t]}\R(\!(t)\!)$. Then the statement is that $\varprojlim$\eqref{1} is $t^{-n}$ times by the inverse limit of the compatible sequence of cohomology classes
$$
p\_{N*}\!\left(\!\frac{\ch\!\Big(\!K^{\frac12}_{\vir}\!\Big)\!\Td\!\(T^{\vir}_M\)}{t^{-n}e(N^{\vir})}\cap\big[M^T\big]^{\vir}\right)\,\in\ H^*\(M\times_T(\C^{N+1}\take\{0\}),\R\).
$$
Using $e(N^{\vir})=\ch(\Lambda\udot(N^{\vir})^\vee)\Td(N^{\vir})$ and $T^{\vir}_M\big|_{M^T}=T^{\vir}_{M^T}+N^{\vir}$ in $K$-theory, these classes are
$$
p\_{N*}\!\left(\!\frac{\ch\!\Big(\!K^{\frac12}_{\vir}\!\Big)}{t^{-n}\ch\(\Lambda\udot(N^{\vir})^\vee\)}\Td\!\(T^{\vir}_{M^T}\)\cap\big[M^T\big]^{\vir}\right).
$$
Let $q\_N\colon M^T\times_T(\C^{N+1}\take\{0\})=M^T\times\PP^N\to\PP^N$ denote the restriction of $p\_N$. Applying \cite[Theorem 3.3]{FG} to $q\_N$ gives
\beq{ans}
\ch\!\left(\!Rq\_{N*}\!\left(\!\frac{K^{\frac12}_{\vir}\big|_{M^T}}{t^{-n}\Lambda\udot(N^{\vir})^\vee}\Td\!\(T^{\vir}_{M^T}\)\cap\big[M^T\big]^{\vir}\right)\right).
\eeq
Since the Chern character of any (virtual) $T$-representation $V$ is $\ch(V)=\chi\_{e^t}(V)$, the upshot is
$$
\chi\_{e^t}(\Ohat)\=\varprojlim\text{\eqref{1.5}}\=t^{-n}\varprojlim\text{\eqref{ans}}\=\chi\_{e^t}\!\left(\!\frac{\cO^{\vir}_{M^T}}{\Lambda\udot(N^{\vir})^\vee}\otimes K^{\frac12}_{\vir}\right).
$$
Substituting $s=e^t$ gives the result.
\end{proof}

Now we have localised to $M^T$ we can use the canonical square root \eqref{sqroot} of Proposition \ref{sqprop} to make an unambiguous definition of our refined invariants.

\begin{ass} \label{ass} We will assume
\begin{enumerate}
\item $M$ is a quasi-projective $T$-variety with projective fixed locus $M^T$,
\item $M$ has a $T$-equivariant symmetric obstruction theory.
\end{enumerate}
By Proposition \ref{sqprop}, this implies the existence of
\begin{itemize}
\item a $T$-equivariant \emph{perfect obstruction theory with $\vd=0$}, and
\item a choice of line bundle on $M^T$ whose square is $K_{M,\vir}|_{M^T}$,
\end{itemize}
and this is all we will actually use. So although it is more elegant to assume (1) and (2) from now on, the reader can replace (2) by the above two conditions instead to get a slightly more general result. We then call the line bundle $K^{1/2}_{M,\vir}\big|_{M^T}$.
\end{ass}

\begin{defn}\label{defdef}
Suppose $M$ satisfies \eqref{ass}. Then we define its refined $K$-theoretic invariant $\chi\_t\(M,\Ohat\)$ to be
\beq{tloc}
\chi\_t\!\left(\!M^T\!,\,\frac{\cO^{\vir}_{M^T}}{\Lambda\udot(N^{\vir})^\vee}\otimes K^{\frac12}_{M,\vir}\Big|_{M^T}\!\right)\,\in\ \Q\(t^{\frac12}\),
\eeq
where $K^{\frac12}_{M,\vir}\big|_{M^T}$ is the canonical choice \eqref{sqroot}.
\end{defn}

Again, when $M$ is compact, \eqref{tloc} specialises to \eqref{twRR} at $t=1$. When $M$ is noncompact (but $M^T$ is compact) and has virtual dimension $\vd=0$ we can still define a cohomological substitute for \eqref{twRR} or \eqref{chivir} (both of which are $\int_{[M]^{\vir}}1$ in the compact case) by localisation \eqref{GP} as follows:
\beq{numT}
\int_{[M^T]^{\vir}}\frac1{e(N^{\vir})}\ \in\ \Q.
\eeq
This lies in $\Q\subset\Q[t,t^{-1}]$ because $\vd(M)=0$ implies that $\rk N^{\vir}=-\vd(M^T)$. Then \eqref{tloc} refines \eqref{numT} even when $M$ is noncompact:

\begin{prop} \label{refine} If $M$ satisfies \eqref{ass} then \eqref{tloc} is a rational function of $t^{1/2}$ with poles only at roots of unity and the origin, but not at 1. We may therefore specialise to $t=1$, where we recover the rational number \eqref{numT}\emph{:}
$$
\chi\_t\!\left.\left(\!M^T\!,\,\frac{\cO^{\vir}_{M^T}}{\Lambda\udot(N^{\vir})^\vee}\otimes K^{\frac12}_{\vir}\right)\right|_{t=1}\=\int_{[M^T]^{\vir}}\frac1{e(N^{\vir})}\ \in\ \Q.
$$
\end{prop}

\begin{proof}
We compute
\begin{eqnarray} \label{fraction}
\chi\_{e^t}\!\left(\!M^T\!,\,\frac{\cO^{\vir}_{M^T}}{\Lambda\udot(N^{\vir})^\vee}\otimes K^{\frac12}_{\vir}\right) &=&
\int_{[M^T]^{\vir}}\frac{\ch\!\Big(\!K^{\frac12}_{\vir}\!\Big)\!\Td\!\(T^{\vir}_{M^T}\)}{\ch\!\(\Lambda\udot N^{\vir}\)^{\!\vee}} \\
&=& \int_{[M^T]^{\vir}}\frac{\ch\!\Big(\!K^{\frac12}_{\vir}\!\Big)\!\Td\!\(T^{\vir}_M\big|_{M^T}\)}{e(N^{\vir})}\,.\hspace{-7mm} \nonumber
\end{eqnarray}
We wish to evaluate this at $t=0$ (i.e. $e^t=1$).
Writing $K^{\frac12}_{\vir}=L\t^w$ for some $w\in\Z[1/2]$,
we expand
\beqa
1/e(N^{\vir}) &=& c_0t^{\vd(M^T)}+c_1t^{\vd(M^T)-1}+\ldots\,, \\
\ch\!\(K^{1/2}_{\vir}\) &=& e^{wt}(1+c), \\ 
\Td\!\(T^{\vir}_{M}\big|_{M^T}\) &=& 1+d_1t+d_2t^2+\ldots\,, \quad\text{so that} \\
\ch\!\(K^{1/2}_{\vir}\)\Td\!\(T^{\vir}_{M}\big|_{M^T}\) &=& 1+e_1t+e_2t^2+\ldots\,,
\eeqa
where $c_i\in H^{2i}(M^T)$ and $c,d_i,e_i\in H^{>0}(M^T)$. Therefore $c_i=0$ for $i\gg0$ and the first series is a finite sum. Consider multiplying it by the last series.
\begin{itemize}
\item The $t^{<0}$ terms of the first series all have coefficients in $H^{>2\vd(M^T)}(M^T)$ (both before and after multiplication by the last series). These integrate to zero against $[M^T]^{\vir}$.
\item The $t^{>0}$ terms of the first series go to 0 at $t=0$, and the same is true when they are multiplied by the last series.
\end{itemize}
So we are left with the $c_{\vd(M^T)}t^0$ term of the first series. When multiplied by any $e_i$ we get a class in $H^{>2\vd(M^T)}(M^T)$ whose integral over $[M^T]^{\vir}$ is zero. So when we multiply $c_{\vd(M^T)}t^0$ by the last series and integrate, we get the same as just (multiplying by 1 and) integrating; this contributes
$$
\int_{[M^T]^{\vir}}c_{\vd(M^T)}\=\int_{[M^T]^{\vir}}\frac1{e(N^{\vir})}\,.
$$
This gives the numerical result claimed. But since it doesn't give the statement about rationality, we now go back to the first line of \eqref{fraction} and expand everything in sight carefully. \medskip

Write $(N^{\vir})^\vee=N^0-N^1$ as a global difference of two $T$-bundles with nonzero weights. Let the Chern roots of $N^0$ and $N^1$ be $x_i+w_it$ and $y_j+v_jt$ respectively, where $w_i,v_j$ are all nonzero integers. Letting $s:=e^t$, we have
$$
\frac1{\ch\!\(\Lambda\udot N^{\vir}\)^{\!\vee}}\=
\frac{\prod_j(1-e^{y_j}e^{tv_j})}{\prod_i(1-e^{x_i}e^{tw_i})}
\=\frac{\prod_j(1-s^{v_j})\Big(1-E(y_j)\frac{s^{v_j}}{1-s^{v_j}}\Big)}{\prod_i(1-s^{w_i})\Big(1-E(x_i)\frac{s^{w_i}}{1-s^{w_i}}\Big)}\,,
$$
where $E(u):=e^u-1$ is the power series $u+u^2/2!+u^3/3!+\cdots$.

If we write $1-q^v=(1-q)[v]'_q$, where $[v]'_q:=1+q+q^2+\cdots+q^{v-1}$ is the quantum integer,\footnote{There is another convention $[v]\_q:=q^{-(v-1)/2}[v]_q'$ for quantum integers \eqref{qq} that we will use in Section \ref{sscase}.} this becomes
\begin{align} \nonumber
(1-s)^{-\rk(N^{\vir})}&\frac{\prod_j[v_j]'_s}{\prod_i[w_i]'_s}
\prod_j\bigg(1-E(y_j)\frac{s^{v_j}}{(1-s)[v_j]'_s}\bigg) \\ \label{long}
\times\prod_i&\bigg(1+E(x_i)\frac{s^{w_i}}{(1-s)[w_i]'_s}+E(x_i)^2\frac{s^{2w_i}}{(1-s)^2([w_i]'_s)^2}+\cdots\bigg).
\end{align}
Now $E(x_i)=x_i+x_i^2/2+x_i^3/3!+\cdots$ lies in cohomological degrees $\ge2$, so, when we multiply out, any $(1-s)^{-j}$ term comes multiplied by a term $a_js^{u_j}$, where $a_j$ has cohomological degree $\ge 2j$ on $M^T$ (and $u_j$ is an integer). And $-\rk(N^{\vir})=\vd(M^T)$, so \emph{on restriction to $[M^T]^{\vir}$} we get
\beq{nak}
(1-s)^{\vd(M^T)}\frac{\prod_j[v_j]'_s}{\prod_i[w_i]'_s}\Big(1+a_1s^{u_1}(1-s)^{-1}+\cdots+a_{\vd}s^{u_{\vd}}(1-s)^{-\vd}\Big)
\eeq
by ignoring all cohomology classes which have degree $\ge2\dim[M^T]^{\vir}=2\vd(M^T)$.

Since $T^{\vir}_{M^T}$ is a fixed complex with trivial $T$ action, $\ch(K_{\vir}^{1/2})\Td(T^{\vir}_{M^T})=s^w(1+\sigma)$ for some $\sigma\in H^{\ge2}(M^T)$ with no $t$ (or $s$) dependence. Multiplying by \eqref{nak} and integrating,  \eqref{fraction} becomes
\beq{pm}
\chi\_s\=s^w\frac{\prod_j[v_j]'_s}{\prod_i[w_i]'_s}\sum_{k=0}^{\vd(M^T)}(1-s)^ks^{u_{\vd-k}}\int_{[M^T]^{\vir}}(1+\sigma)\wedge a_{\vd-k}\,.
\eeq
Thus $\chi\_s$ is a rational function of $s^{1/2}$ with poles only at roots of unity and possibly zero, but not 1, as required. \medskip

Since we've got this far we may as well use \eqref{pm} to give another derivation of the evaluation at $s=1$. This gives
$$
\chi\_1\=\frac{\prod_jv_j}{\prod_iw_i}\int_{[M^T]^{\vir}}a_{\vd}\,.
$$
This integral sees only the part of $a_{\vd}$ which has degree precisely $2\!\;\vd$, so only the degree 2 parts $x_i,\,y_j$ of $E(x_i),\,E(y_j)$ in \eqref{long} contribute to it. So replacing $E(x_i),\,E(y_j)$ by $x_i,\,y_j$ in \eqref{long}, it becomes the cohomological degree $\vd$ part of
$$
(1-s)^{\vd}\frac{\prod_j[v_j]'_s}{\prod_i[w_i]'_s}\frac{\prod_j\Big(1-\frac{y_js^{v_j}}{(1-s)[v_j]'_s}\Big)}{\prod_i\Big(1-\frac{x_is^{w_i}}{(1-s)[w_i]'_s}\Big)}\=(1-s)^{\vd}\frac{\prod_j\Big([v_j]'_s-\frac{y_js^{v_j}}{(1-s)}\Big)}{\prod_i\Big([w_i]'_s-\frac{x_is^{w_i}}{(1-s)}\Big)}\,.
$$
To evaluate at $s=1$, we now reuse $t$ to denote $s-1$ and take the coefficient of $t^0$ in the Laurent expansion (in $t^{-1}$!) of
$$
(-t)^{\vd}\frac{\prod_j\Big(v_j+\frac{y_js^{v_j}}t\Big)}{\prod_i\Big(w_i+\frac{x_is^{w_i}}t\Big)}\=\frac{\prod_j(-v_jt-y_j)}{\prod_i(-w_it-x_i)}\=\frac1{e(N^{\vir})}\,.
$$
When we expand $1/e(N^{\vir})$ as a Laurent series in $t^{-1}$ the coefficient of $t^i$ lies in $H^{2\vd(M^T)-2i}(M^T)$. Therefore integrating over $[M^T]^{\vir}$ takes only the $t^0$ term. We conclude again that
\[
\chi\_1\=\int_{[M^T]^{\vir}}\frac1{e(N^{\vir})}\,.\qedhere
\]
\end{proof}

\begin{prop} \label{tt-1} The refinement \eqref{tloc} is a rational function of $t^{\frac12}$ (with poles only at roots of unity and the origin) invariant under $t^{\frac12}\leftrightarrow t^{-\frac12}$, and is deformation invariant.
\end{prop}

\begin{proof}
All that is left to prove is the invariance under $t^{\frac12}\leftrightarrow t^{-\frac12}$. This follows from the ``weak virtual Serre duality" of \cite[Proposition 3.13]{FG} on $M^T$:
\begin{align*}
\chi\_{t^{-1}}&\!\left(\!M^T,\,\frac{\cO^{\vir}_{M^T}}{\Lambda\udot(N^{\vir})^\vee}\otimes K^{\frac12}_{M,\vir}\right) \\
&\=(-1)^{\vd(M^T)}\chi\_t\!\left(\!M^T\!,\,\frac{\cO^{\vir}_{M^T}}{\Lambda\udot N^{\vir}}\otimes K^{-\frac12}_{M,\vir}\otimes K_{M^T\!,\vir}\right) \\
&\=(-1)^{\rk(N^{\vir})}\chi\_t\!\left(\!M^T\!,\,\frac{\cO^{\vir}_{M^T}}{\Lambda\udot N^{\vir}}\otimes K^{\frac12}_{M,\vir}\otimes\det N^{\vir}\right) \\
&\=\chi\_t\!\left(\!M^T\!,\,\frac{\cO^{\vir}_{M^T}}{\Lambda\udot(N^{\vir})^\vee}\otimes K^{\frac12}_{M,\vir}\right)\!,
\end{align*}
where the last equality is from the identity
\beq{indenti}
\Lambda\udot(N^{\vir})^\vee\ \cong\ (-1)^{\rk(N^{\vir})}\Lambda\udot N^{\vir}\otimes\det(N^{\vir})^*. \qedhere
\eeq
\end{proof}

\subsection{Shifted cotangent bundles} When $M$ is the $(-1)$-shifted cotangent bundle $T^*[-1]M^T$ of a quasi-smooth derived projective scheme $M^T$, with the obvious $T=\C^*$ action on the fibres, it has a symmetric perfect obstruction theory and the refined invariant \eqref{tloc} simplifies. Letting $p$ denote the projection $M\to M^T$ we get the exact triangle
\beq{tri}
p^*\(T_{\!M^T}^{\vir}\)\!\;^\vee\!\otimes\t\;[-1]\To T_M^{\vir}\rt{p_*}p^*T_{\!M^T}^{\vir},
\eeq
where $\t$ denotes the standard weight 1 representation of $T$. In particular
$K_{M,\vir}\=p^*(K_{M^T,\vir})^2\otimes\t^{\vd(M^T)}$ and the canonical choice \eqref{sqroot} of square root is just
\beq{kver}
K_{M,\vir}^{\frac12}\ :=\ K_{M^T\!\!,\,\vir}\,\t^{\vd(M^T)/2}.
\eeq
On the zero section $M^T\subset M$ the exact triangle \eqref{tri} gives
$$
N^{\vir}\=\(T_{\!M^T}^{\vir}\)\!\;^\vee\!\otimes\t\;[-1]\=\LL^{\vir}_{M^T}\;\t\,[-1].
$$
Therefore, using both \eqref{kver} and the identity \eqref{indenti}, we have in localised $K$-theory,
\begin{multline*}
\frac1{\Lambda\udot\(N^{\vir}\)^\vee}\otimes K^{\frac12}_{M,\vir}\=
(-1)^{\rk(N^{\vir})}\frac{\det(N^{\vir})\otimes K_{M^T\!\!,\,\vir}\,\t^{\vd(M^T)/2}}{\Lambda\udot\;N^{\vir}} \\
\= (-1)^{\vd(M^T)}K_{M^T\!\!,\,\vir}^{-1}\,\t^{-\vd(M^T)}\otimes K_{M^T\!\!,\,\vir}\,\t^{\vd(M^T)/2}\otimes\Lambda\udot(\LL^{\vir}_{M^T}\t)
\\ \=
(-1)^{\vd(M^T)}\t^{-\vd(M^T)/2}\Lambda\udot\(\LL_{M^T}^{\vir}\t\).
\end{multline*}
Substituted into \eqref{tloc}, this gives the following result.

\begin{prop}\label{tstar-1} If $M^T$ is a quasi-smooth derived projective scheme, the $K$-theoretic refined invariant \eqref{tloc} of $M=T^*[-1]M^T$ is
\begin{eqnarray}
\chi\_t\(M,\Ohat\) &=& (-1)^{\vd(M^T)}t^{-\frac{\vd(M^T)}2}\sum_i(-1)^i\chi\!\left(\cO^{\vir}_{M^T}\otimes\Lambda^i\LL^{\vir}_{M^T}\right)t^i \nonumber \\
&=:\!& (-1)^{\vd(M^T)}t^{-\frac{\vd(M^T)}2}\chi^{\vir}_{-t}\(M^T\), \label{shifted}
\end{eqnarray}
where $\chi^{\vir}_y$ is the virtual $\chi\_y$-genus of Fantechi-G\"ottsche \cite{FG}. Specialising to $t=1$ gives the signed virtual Euler characteristic studied in \cite{JT}.\hfill$\square$
\end{prop}

\section{$K$-theoretic invariants from degeneration loci}\label{degensec}
Fix a map of bundles
$$
\sigma\,:\ E_0\To E_1
$$
of ranks $r_0,r_1$ over a smooth ambient space $X$. We suppose for simplicity\footnote{More general degeneracy loci were treated in \cite{GT1, GT2}.} that $\dim\ker(\sigma|_x)\le1$ for all points $x\in X$.
Then we let
$$
D(\sigma)\ :=\ \big\{x\in X\ \colon\ \sigma|_x\ \mathrm{is\ not\ injective}\big\}
$$
denote the degeneracy locus of $\sigma$. Its scheme structure is defined by the vanishing ideal of $\Lambda^{r_0}\sigma$ --- the ideal generated by the $r_0\times r_0$ minors of $\sigma$.

Furthermore it is shown in \cite{GT1} that $D(\sigma)$ inherits  a \emph{perfect obstruction theory} by seeing it as the vanishing locus of the composition
\beq{compo}
\cO_{\PP(E_0)}(-1)\Into q^*E_0\rt{q^*\sigma}q^*E_1 \quad\mathrm{on}\ \ \PP(E_0)\rt{q}X.
\eeq
This perfect obstruction theory depends only on the complex $E_0\to E_1$ up to quasi-isomorphism, and endows $D(\sigma)$ with a virtual cycle of codimension $r_1-r_0+1$ whose pushforward to $X$ is described in \cite{GT1} by the Thom-Porteous formula as
\beq{Tpo}
\iota_*\big[D(\sigma)\big]^{\vir}\=c_{r_1-r_0+1}(E_1-E_0)\cap[X].
\eeq

\subsection{$K$-theoretic degeneracy loci}
The above perfect obstruction theory induces a virtual structure sheaf $\cO^{\vir}_{D(\sigma)}$ on the degeneracy locus by \eqref{vss}.

The $K$-theoretic analogue of the Thom-Porteous formula is the \emph{Eagon-Northcott complex} of $\sigma$. When $D(\sigma)$ has the correct codimension, this complex is well known to resolve $\cO_{D(\sigma)}$. Here we show that even when it has the wrong codimension, the $K$-theory class of the Eagon-Northcott complex is $\cO^{\vir}_{D(\sigma)}\in K_0(X)$.

\begin{thm}\label{3.3}
The pushforward of $\cO^{\vir}_{D(\sigma)}$ to $X$ has $K$-theory class
$$
\cO_X-\det(E_0-E_1)\otimes\bigwedge\!\!^r(E_1-E_0),
$$
where $r=\rk E_1-\rk E_0$.
\end{thm}

\begin{proof}
The composition \eqref{compo} defines a section $\widetilde\sigma\in\Gamma(q^*E_1(1))$ which cuts out $D(\sigma)\subset\PP(E_0)$. This defines the virtual structure sheaf
$$
\cO_{D(\sigma)}^{\vir}\=L\iota_0^*\;\cO_{C_{\widetilde\sigma}},
$$
where $\iota\_0\colon D(\sigma)\into q^*E_1(1)\big|_{D(\sigma)}$ is the zero section and
$$
C_{\widetilde\sigma}\ \subset\ q^*E_1(1)\big|_{D(\sigma)}
$$
is the cone over $D(\sigma)$ defined by $\widetilde\sigma$. This is the flat limit, as $t\to\infty$, of the graphs $\Gamma_{t\widetilde\sigma}\subset q^*E_1(1)$. Therefore, in $K$-theory,
\beq{eurostar}
j_*\;\cO_{D(\sigma)}^{\vir}\=L\iota^*\;\cO_{\Gamma_{\widetilde\sigma}},
\eeq
where $j\colon D(\sigma)\into\PP(E_0)$ and $\iota\colon\PP(E_0)\into q^*E_1(1)$ is the zero section.

Suppressing some obvious pullback maps for clarity, $\Gamma_{\widetilde\sigma}$ is cut out of the total space of $q^*E_1(1)$ by the section $\widetilde\sigma-\sigma_{\mathrm{taut}}$ of the pullback of $q^*E_1(1)$. This induces a Koszul resolution of $\cO_{\Gamma_{\widetilde\sigma}}$ on the total space of $q^*E_1(1)$. Applying $L\iota^*$ restricts this to the zero section, where $\sigma_{\mathrm{taut}}$ is zero. Thus the right hand side of \eqref{eurostar} is the Koszul complex $\(\Lambda\udot(q^*E_1(1))^*,\widetilde\sigma\)$ (which is not in general exact, since $\Gamma_{\widetilde\sigma}$ does not generally intersect the zero section transversally). Thus
\beq{thus}
j_*\;\cO_{D(\sigma)}^{\vir}\=\sum_{i=0}^{r\_1}(-1)^i\Lambda^i(q^*E_1^*)(-i),
\eeq
where $r_i:=\rk(E_i)$.

Finally we push down to $X$, using (by Serre duality)
$$
Rq_*\;\cO_{\PP(E_0)}(-i)\=\left\{\!\!\begin{array}{cl}\cO_X & i=0, \\
0& 0<i<r\_0, \\
\Sym^{i-r\_0}E_0\otimes\det E_0\;[1-r_0] & i\ge r\_0.
\end{array}\right.
$$
Thus, applying $Rq_*$ to \eqref{thus}, we find the pushforward of $\cO_{D(\sigma)}^{\vir}$ to $K(X)$ is
\begin{multline*}
\cO+\sum_{i=r\_0}^{r\_1}(-1)^{i+1-r\_0}\Lambda^i(E_1^*)\otimes\Sym^{i-r\_0}E_0\otimes\det E_0 \\
\=\cO-\Big(\sum_{i=r\_0}^{r\_1}(-1)^{i-r\_0}\Lambda^{r\_1-i}E_1\otimes\Sym^{i-r\_0}E_0\Big)\!\otimes\det E_0\otimes\det E_1^* \\
\!\!\!\!\=\cO-\Big(\sum_{j=0}^r(-1)^{j}\Lambda^{r-j}E_1\otimes\Sym^jE_0\Big)\!\otimes\det(E_0-E_1) \\
\=\cO-\bigwedge\!\!^r(E_1-E_0)\otimes\det(E_0-E_1),\hspace{37mm}
\end{multline*}
where $r=r\_1-r\_0=\rk(E_1-E_0)$.
\end{proof}

There are different formulae for (more general) $K$-theoretic degeneracy loci in \cite{HIMN, A}, essentially given by the Thom-Porteous formula with cohomological Chern classes replaced by $K$-theoretic Chern classes. By some algebraic identities these formulae are surely equivalent to the Eagon-Northcott formula above. Therefore, by Theorem \ref{3.3}, those formulae also describe the pushforward of the virtual structure sheaf of a degeneracy locus.

\subsection{Application to Vafa-Witten theory}
In \cite{GT1, GT2} it was shown how some of the monopole components of the Vafa-Witten $T$-fixed point set can be described as degeneracy loci, at the level of both their scheme structures and virtual cycles. We briefly review the simplest examples --- 2-step nested punctual Hilbert schemes of a smooth projective surface $S$,
$$
S^{[n_1,n_2]}\ :=\ \big\{I_1\subseteq I_2\subseteq\cO_S\ \colon\ \mathrm{length}\;(\cO_S/I_i)=n_i\big\}.
$$
For more details and more general results see \cite{GT1, GT2}.

While $S^{[n_1,n_2]}$ is in general singular, it lies in the smooth ambient space $S^{[n_1]}\times S^{[n_2]}$ as the set of points $(I_1,I_2)$ for which there is a nonzero map $\Hom_S(I_1,I_2)\ne0$.
To write this scheme theoretically, let $\pi\colon S^{[n_1]}\times S^{[n_2]}\times S\to S^{[n_1]}\times S^{[n_2]}$ be the projection down $S$, and let $\cI_1,\cI_2$ denote the (pullbacks of the) universal ideal sheaves on $S^{[n_1]}\times S^{[n_2]}\times S$. Consider the complex of vector bundles
\beq{complex}
R\hom_\pi(\cI_1,\cI_2) \quad\mathrm{over}\quad S^{[n_1]}\times S^{[n_2]},
\eeq
which, restricted to a point $(I_1,I_2)$, computes $\Ext^*_S(I_1,I_2)$. When $p_g(S)=0$ this complex can be made 2-term,\footnote{When $p_g(S)>0$ it is shown in \cite[Section 6]{GT1} how to remove $H^2(\cO_S)$ from the complex to give the same result.}
$$
R\hom_\pi(\cI_1,\cI_2)\quad\simeq\quad E_0\rt\sigma E_1.
$$
Then the degeneracy locus $D(\sigma)$ is (scheme-theoretically) $S^{[n_1,n_2]}$, and the construction \eqref{compo} endows it with a perfect obstruction theory of dimension $n_1+n_2$ which is independent of the choices of $E_0,E_1$ (depending only on the $K$-theory class of their difference $R\hom_\pi(\cI_1,\cI_2)$.)

By \cite[Section 10]{GT1} this perfect obstruction theory and the Vafa-Witten perfect obstruction theory of \cite{TT1} have virtual tangent bundles which agree in $K$-theory.\footnote{It is shown in \cite[Section 10]{GT1} that the two perfect obstruction theories $E\udot\to\LL_{S^{[n_1,n_2]}}$ have isomorphic virtual cotangent bundles $E\udot$, but not that their maps to the cotangent complex $\LL$ are the same (though they surely are).}
Therefore the degeneracy locus virtual cycle coincides which the one coming from Vafa-Witten theory \cite{TT1} or reduced DT theory \cite{GSY1, GSY2} when $h^1(\cO_S)=0$. And when $h^1(\cO_S)>0$, the complex $R\hom_\pi(\cI_1,\cI_2)$ can be modified \cite[Section 6]{GT1}, or one can work relative to $\Pic(S)$ \cite[Section 9]{GT1}, to get the same result.

The Thom-Porteous formula \eqref{Tpo} then calculates the pushforward to $S^{[n_1]}\times S^{[n_2]}$ of the Vafa-Witten virtual cycle on $S^{[n_1,n_2]}$. In fact \cite{GT1, GT2} deal with more complicated nested Hilbert schemes of points and curves on $S$. For points, the result is the following.

\begin{thm}\cite{GT1}
The pushforward of $\big[S^{[n_1,n_2,\ldots,n_r]}\big]^{\vir}$ to $S^{[n_1]}\times\cdots\times S^{[n_r]}$ equals the product of Carlsson-Okounkov classes
$$
c_{n_1+n_2}\big(R\hom_{\pi}(\cI_1,\cI_2)[1]\big)\cup\cdots\cup
c_{n_{r-1}+n_r}\big(R\hom_{\pi}(\cI_{r-1},\cI_r)[1]\big)
$$
in $A_{n_1+n_r}\big(S^{[n_1]}\times\cdots\times S^{[n_r]}\big)$.
\end{thm}

For the $K$-theoretic analogue, we assume for simplicity that $H^{\ge1}(\cO_S)=0$ so we do not have to modify the complex \eqref{complex}.
Theorem 3.2 of \cite{Th} gives a formula for any virtual structure sheaf $\cO^{\vir}_M$ which shows that it depends only on $M$ and the $K$-theory class of the virtual tangent bundle $T^{\vir}_M$. Since the degeneracy locus construction induces the same virtual tangent bundle on $S^{[n_1,n_2]}$ as Vafa-Witten theory, the two virtual structure sheaves induced by \eqref{vss} are equal.

Therefore Theorem \ref{3.3} describes the Vafa-Witten $K$-theoretic virtual cycle as follows. 

\begin{cor} When $H^{\ge1}(\cO_S)=0$ 
the pushforward of the Vafa-Witten virtual structure sheaf $\cO_{S^{[n_1,n_2]}}^{\vir}$ to $S^{[n_1]}\times S^{[n_2]}$ is the $K$-theory class of
$$
\cO_{S^{[n_1]}\times S^{[n_2]}}-\det\!\(R\hom_\pi(\cI_1,\cI_2)\)\otimes\bigwedge\!\!^r\(R\hom_\pi(\cI_1,\cI_2)[1]\),
$$
where $r=n_1+n_2-1$. Letting $\mathsf{CO}:=\cO-R\hom_\pi(\cI_1,\cI_2)$, this can be written --- in topological $K$-theory at least --- as
$$
\Lambda\udot(\mathsf{CO}^*)\=\sum_{i=0}^{n_1+n_2}(-1)^i\Lambda^i(\mathsf{CO}^*).
$$
\end{cor}

For a general surface $S$, the techniques of \cite{GT1} can be used to split off $H^{\ge1}(\cO_S)$ from this complex. The upshot is the same result, except with the complex $R\hom_\pi(\cI_1,\cI_2)$ replaced by $R\hom_\pi(\cI_1,\cI_2)+R^1\pi_*\cO-R^2\pi_*\cO$.

\begin{proof}
Theorem \ref{3.3} gives the first result immediately. For the second we use \cite[Equation 4.27]{GT2}, which shows the Carlsson-Okounkov $K$-theory class $\mathsf{CO}:=\cO-R\hom_\pi(\cI_1,\cI_2)$ is represented by an honest rank $r+1=n_1+n_2$ vector bundle (instead of a difference of vector bundles) on an affine bundle over $S^{[n_1]}\times S^{[n_2]}$. Therefore, writing the pullback of
$R\hom_\pi(\cI_1,\cI_2)[1]$ as the difference of bundles $\mathsf{CO}-\cO$, we find
\begin{multline*}
\hspace{3mm} \cO-\det\!\(R\hom_\pi(\cI_1,\cI_2)\)\otimes\bigwedge\!\!^r\(R\hom_\pi(\cI_1,\cI_2)[1]\) \\
\=\cO-\det(\mathsf{CO}^*)\otimes\sum_{i=0}^r(-1)^i\Lambda^{r-i}(\mathsf{CO})\otimes\Sym^i\cO \\
\=\cO-\sum_{i=0}^r(-1)^i\Lambda^{i+1}(\mathsf{CO}^*)\=\Lambda\udot( \mathsf{CO}^*)\hspace{1cm}
\end{multline*}
upstairs on the affine bundle. Since this is homotopic to the base $S^{[n_1]}\times S^{[n_2]}$, the result follows in topological $K$-theory (which, by Riemann-Roch, is enough for computing holomorphic   Euler characteristics). Therefore in this situation we get a more familiar Koszul resolution description of the virtual structure sheaf.
\end{proof}

These results can be plugged into the definition \eqref{eqref} of refined Vafa-Witten invariants below to calculate monopole locus contributions in terms of $K$-theory classes on smooth ambient spaces like $S^{[n_1]}\times S^{[n_2]}$. (This requires lifts of the $K$-theory classes $N^{\vir}$ and $K_{\vir}^{1/2}$ \eqref{sqroot} from $S^{[n_1,n_2]}$ to $S^{[n_1]}\times S^{[n_2]}$; these exist since they can be written in terms of the universal sheaves $\cI_{\cZ_1},\,\cI_{\cZ_1}$ and $R\hom_\pi$\,s between them, all of which extend.)

\section{Refined Vafa-Witten invariants: stable case}\label{VWsec}
Fix a smooth complex poplarised surface $(S,\cO_S(1))$, a rank $r>0$, Chern classes $c_1,\,c_2$ and a line bundle $L$ on $S$ with $c_1(L)=c_1$. We use the notation from \cite{TT1}; in particular $\cN^\perp_{r,L,c_2}$ denotes the moduli space of Gieseker semistable Higgs pairs $(E,\phi)$ on $S$, where
$E$ is a rank $r$ torsion-free sheaf on $S$ with $c_2(E)=c_2$ and
$$
\det E\cong L, \qquad \phi\ \in\ \Hom(E,E\otimes K_S)\_0\;.
$$
By the spectral construction \cite[Section 2]{TT1}, $(E,\phi)$ is equivalent to a Gieseker semistable compactly supported torsion sheaf $\cE_\phi$ on
$$
X\ :=\ K_S\rt\rho S.
$$
This makes $\cN^\perp_{r,L,c_2}$ the moduli space of Gieseker semistable compactly supported torsion sheaves $\cE$ on $X$ with the right Chern classes, such that the ``centre of mass'' of $\cE$ on each $K_S$ fibre of $\rho\colon X\to S$ is zero, and $\rho_*\;\cE\cong L$.

When $(r,c_1,c_2)$ are chosen so that semistability implies stability\footnote{The general case is handled in \cite{TT2}.}, $\cN^\perp_{r,L,c_2}$ carries a natural symmetric perfect obstruction theory \cite[Theorem 6.1]{TT1} supported in degrees $[-1,0]$. It is noncompact in general, but has a $T=\C^*$ action scaling the fibres of $K_S$ (or, equivalently, scaling the Higgs field $\phi$). The fixed locus $(\cN^\perp_{r,L,c_2})^T$ is compact, so in \cite{TT1} Vafa-Witten invariants are defined by virtual (cohomological) localisation
\beq{defVW}
\mathsf{VW}_{r,c_1,c_2}\ :=\ \int_{\big[(\cN_{r,L,c_2}^\perp)^T\big]^{\vir\ }}\frac1{e(N^{\vir})}\ \in\ \Q,
\eeq
where $L$ is any line bundle on $S$ with $c_1(L)=c_1$. This lies in $\Q\subset\Q[t,t^{-1}]$ because $\vd(\cN^\perp_{r,L,c_2})=0$.

The perfect obstruction theory gives us a virtual structure sheaf, and its symmetry gives us a canonical square root of the virtual canonical bundle  by Proposition \ref{sqprop}. 

\begin{defn} \label{defref}
For $r,c_1,c_2$ such that semistability implies stability,
the refined Vafa-Witten invariants of $(S,\cO_S(1))$ are defined by \eqref{tloc}\emph{:}
\begin{eqnarray}\nonumber
\mathsf{VW}_{r,c_1,c_2}(t) &:=& \chi\_t\Big(\cN_{r,L,c_2}^\perp,\widehat\cO^{\vir}_{\cN_{r,L,c_2}^\perp}\Big) \\
&=& \chi\_t\!\left(\!(\cN_{r,L,c_2}^\perp)^T\!,\,\frac{\cO^{\vir}_{(\cN_{r,L,c_2}^\perp)^T}}{\Lambda\udot(N^{\vir})^\vee}\otimes K^{\frac12}_{\vir}\right)\,\in\ \Q(t^{1/2}). \label{eqref}
\end{eqnarray}
By Proposition \ref{refine} this refines \eqref{defVW}, specialising to $\mathsf{VW}_{r,c_1,c_2}$ at $t=1$.
\end{defn}

\noindent\textbf{First fixed locus.}
The invariant \eqref{defVW} and its refinement \eqref{eqref} have contributions from the connected components of the fixed locus $(\cN_{r,L,c_2}^\perp)^T$. The first of these is the ``instanton branch" where $\phi=0$.

This is just the moduli space $\cM_{r,L,c_2}$ of stable sheaves of fixed determinant $L$ on $S$. By \cite[Equation 7.3]{TT1} this locus contributes the Fantechi-G\"ottsche virtual signed Euler characteristic
\begin{align}
\int_{[\cM_{r,L,c_2}]^{\vir}}c\_{\vd(\cM_{r,L,c_2})}\(\LL^{\vir}_{\cM_{r,L,c_2}}\)& \nonumber \\ =\
(-1)^{\vd(\cM_{r,L,c_2})}&\int_{[\cM_{r,L,c_2}]^{\vir}}c\_{\vd(\cM_{r,L,c_2})}\(T^{\vir}_{\cM_{r,L,c_2}}\) \label{evir} \\ &\quad=\
(-1)^{\vd(\cM_{r,L,c_2})}e^{\vir}(\cM_{r,L,c_2}) \nonumber
\end{align}
of $\cM_{r,L,c_2}$ to $\mathsf{VW}_{r,c_1,c_2}$ \eqref{defVW}. It is an integer, with an obvious refinement given (up to a sign) by the virtual $\chi\_t$-genus of Fantechi-G\"ottsche \cite{FG}, as studied in \cite{GK1}. We check now that this is what the refined Vafa-Witten invariant \eqref{eqref} indeed gives.

An open neighbourhood in $\cN^\perp_{r,L,c_2}$ of the instanton branch $\cM_{r,L,c_2}$ is isomorphic to its own $(-1)$-shifted cotangent bundle $T^*[-1]\cM_{r,L,c_2}$. (It is the neighbourhood consisting of those pairs $(E,\phi)$ for which $E$ is itself Gieseker stable. At the level of points this says that the Higgs fields take values in the dual $(\Ext^2(E,E)_0)^*\cong\Hom(E,E\otimes K_S)$ of the obstruction space of $E$.) By \eqref{shifted}, then, its contribution to $\mathsf{VW}(t)$ \eqref{eqref} is
$$
(-1)^{\vd(\cM_{r,L,c_2})}t^{-\frac{\vd(\cM_{r,L,c_2})}2}\chi^{\vir}_{-t}\(\cM_{r,L,c_2}\).
$$
Specialising to $t=1$ gives
$$
(-1)^{\vd(\cM_{r,L,c_2})}e^{\vir}\(\cM_{r,L,c_2}\),
$$
which is \eqref{evir}.\medskip

\noindent\textbf{Second fixed locus.} The other fixed loci have nilpotent $\phi\ne0$; following \cite{DPS, GK1} we call their union $\cM_2$ the ``monopole branch". For $K_S$ negative in a suitable sense the stability condition forces the second fixed locus to be empty (a ``vanishing theorem" holds). So we do some very elementary preliminary calculations on $\cM_2$ for general type surfaces, refining the first calculations of \cite[Section 8]{TT1}.

\subsection{\for{toc}{Some calculations for $K_S>0$}\except{toc}{Some calculations for $\mathbf{K_S>0}$}}
Let $(S,\cO_S(1))$ be a smooth, connected polarised surface with
\begin{itemize}
\item $h^1(\cO_S)=0$, and
\item a smooth nonempty connected canonical divisor $C\in|K_S|$, such that
\item $L=\cO_S$ is the only line bundle satisfying $0\le\deg L\le\frac12\deg K_S$,
\end{itemize}
where degree is defined by $\deg L=c_1(L)\cdot c_1(\cO_S(1))$. Then in \cite[Lemma 8.3]{TT1} it is shown that $(\cN^\perp_{2,K_S,n})^T$ is the disjoint union of $\cM_{2,K_S,n}$ and the nested Hilbert schemes
\beq{nested}
\cM_2\ \cong\ \bigsqcup_{i=0}^{\lfloor n/2\rfloor}S^{[i,\,n-i]}
\eeq
of subschemes $Z_1\subseteq Z_0\subset S$ of lengths
$$
|Z_1|\ =\ i,\qquad |Z_0|+|Z_1|\ =\ n.
$$
We call the components with $i=0$ \emph{horizontal} and, at the other extreme, the components with $i=n/2$ \emph{vertical}.
\medskip

\noindent\textbf{Horizontal loci and $n\le1$ case.}
We start with the horizontal loci, so $Z_1=\emptyset$ and $\cM_2\cong S^{[n]}$. Here a point $Z_0\in S^{[n]}$ corresponds in $\cN^\perp_{2,K_S,n}$ to the torsion sheaf
$I_{Z_0\subset 2S}$ on $X=K_S$, where $2S\subset X$ is the twice-thickened zero section defined by the ideal $I_{S\subset X}^2$. In \cite[Section 8.2]{TT1} it is shown that the fixed obstruction bundle over this $S^{[n]}$ is $\(K_S^{[n]}\)^*$. It follows that
\beq{tvir}
T^{\vir}_{S^{[n]}}\=T_{S^{[n]}}-\(K_S^{[n]}\)^*.
\eeq

It also follows that $\cO^{\vir}_{S^{[n]}}=\Lambda\udot\(K_S^{[n]}\)$. In $K$-theory this has the same class as the sheaf of dgas given by inserting a differential,
\beq{kosz}
\cO_{S^{[n]}}\rt{s^{[n]}}K_S^{[n]}\rt{\wedge s^{[n]}}\Lambda^2K_S^{[n]}\rt{\wedge s^{[n]}}\ldots\rt{\wedge s^{[n]}}\Lambda^nK_S^{[n]},
\eeq
where $s\in H^0(K_S)$ cuts out the smooth connected canonical divisor $C\subset S$ and $s^{[n]}$ is the induced section of $K_S^{[n]}$ on $S^{[n]}$. Since this cuts out $C^{[n]}\subset S^{[n]}$ which is smooth of codimension $n$, \eqref{kosz} is an exact Koszul resolution quasi-isomorphic to its cokernel:
\beq{ovircalc}
\cO_{S^{[n]}}^{\vir}\=(-1)^n\bigg[\Lambda^n\(K_S^{[n]}\)\Big|_{C^{[n]}}\bigg]\ \in\ K^0\(S^{[n]}\).
\eeq
Moreover the $K$-theory class of the virtual normal bundle is computed in \cite[Section 8.3]{TT1} to be
\beq{Nvircalc}
\big[K_S^{[n]}\big]\t\ +\ 
(\t^2)^{\oplus P_2}\ -\ 
\big[(K_S^2)^{[n]}\big]\t^2\ -\ 
(\t^{-1})^{\oplus P_2}\ +\ 
\big[\big((K_S^2)^{[n]}\big)^*\big]\t^{-1}\ -\ 
\big[\Omega_{S^{[n]}}\big]\t.
\eeq
Here $$P_2\=h^0(K_S^2)\=p_g(S)+g,$$ where $g=c_1(K_S)^2+1$ is the genus of $C$.

Combining \eqref{Nvircalc} with \eqref{tvir} determines the virtual canonical bundle:
$$
K_{\vir}\=K_{S^{[n]}}^2\otimes\det\!\(K_S^{[n]}\)^{-2}\otimes\det\!\Big(\!\(K_S^2\)^{[n]}\Big)^2\otimes\t^{4n-3P_2}.
$$
The choice \eqref{sqroot} of square root is
\beq{Khalfcalc}
K^{\frac12}_{\vir}\ :=\ K_{S^{[n]}}\otimes\det\!\(K_S^{[n]}\)^*\otimes\det\!\Big(\!\(K_S^2\)^{[n]}\Big)\otimes\t^{2n-\frac32P_2}.
\eeq
With this, we are ready to calculate in the simplest, $n=0,1$ and $2$ cases. Even here the calculation will be somewhat involved.
\medskip

\noindent{\bf$n=0$ case.} Here $\cM_2\cong S^{[0]}$ is a reduced point, corresponding to the torsion sheaf $\cO_{2S}:=\cO_X/\cI_S^2$ on $X=K_S$. From above,
$$
(N^{\vir})^\vee\=(\t^{-2})^{\oplus P_2}-\t^{\oplus P_2}, \quad K_{\vir}^{\frac12}\=\t^{-\frac32P_2}
\quad\text{and}\quad \cO^{\vir}_{S^{[0]}}=\cO_{S^{[0]}}.
$$
Therefore we calculate the contribution of $S^{[0]}$ to \eqref{eqref} as
\beq{zero}
\chi\_t\=t^{-\frac32P_2}\frac{(1-t)^{P_2}}{(1-t^{-2})^{P_2}}
\=\frac{(-1)^{P_2}}{(t^{\frac12}+t^{-\frac12}\)^{P_2}}\=
\frac{(-1)^{P_2}}{[2]^{P_2}_t}\,,
\eeq
where $[2]\_t$ is the quantum integer \eqref{qq}.
\medskip

\noindent{\bf$n=1$ case.}
Combining \eqref{Khalfcalc}
$$
K_{\vir}^{\frac12}\=K_S^2\otimes\t^{2-\frac32P_2}
$$
with \eqref{ovircalc} we see the contribution of $\cM_2=S^{[1]}=S$ to \eqref{eqref} is
\beq{sofar}
-t^{2-\frac32P_2}\,\chi\_t\!\left(\!C,\left.\frac{K_S^3}{\Lambda\udot(N^{\vir})^\vee}\right|_C\right)\!,
\eeq
where, by \eqref{Nvircalc},
$$
(N^{\vir})^\vee\=K_S^*\t^{-1}\ +\ K_S^2\t\ +\ 
(\t^{-2})^{\oplus P_2}\ -\ 
K_S^{-2}\t^{-2}\ -\ T_S\t^{-1}\ -\ \t^{\oplus P_2}.
$$
Since $T_S|\_C=T_C\oplus\cO_C(C)$ in $K$-theory we can write
$$
(N^{\vir})^\vee\big|_C\=\sum_iL_i\t^{w_i}-\sum_iM_i\t^{v_i},
$$
where the $(L_i,w_i)$ are
\beq{Li}
\(K_S^*\big|_C\;,-1\),\ \(K_S^2\big|_C\;,1\) \mathrm{\ and\ }P_2\mathrm{\ copies\ of\ } (\cO_C,-2),
\eeq
and the $(M_i,v_i)$ are
\beq{Mi}
\(K_S^{-2}\big|_C\;,-2\),\ \(T_C,-1\),\ \(\cO_C(C),-1\) \mathrm{\ and\ }P_2\mathrm{\ copies\ of\ } (\cO_C,1).
\eeq
Therefore
\begin{eqnarray}
\ch\!\left(\!\frac1{\Lambda\udot(N^{\vir})^\vee|\_C\!}\right)\!\!\! &=&
\!\frac{\prod_i(1-e^{m_i+v_it})}{\prod_i(1-e^{\ell_i+w_it})}
\=\frac{\prod_i(1-(1+m_i)e^{v_it})}{\prod_i(1-(1+\ell_i)e^{w_it})} \nonumber \\
&=& \!\frac{\prod_i(1-e^{v_it})}{\prod_i(1-e^{w_it})}\cdot\frac{\prod_i\Big(1-m_i\frac{e^{v_it}}{1-e^{v_it}}\Big)}{\prod_i\Big(1-\ell_i\frac{e^{w_it}}{1-e^{w_it}}\Big)} \nonumber \\
&=& \!\frac{\prod_i(1-e^{v_it})}{\prod_i(1-e^{w_it})}\prod_i\!\left(\!1-m_i\frac{e^{v_it}}{1-e^{v_it}}\!\right)\!\prod_i\!\left(\!1+\ell_i\frac{e^{w_it}}{1-e^{w_it}}\!\right)\!, \label{lm}
\end{eqnarray}
where $m_i:=c_1(M_i),\ \ell_i:=c_1(L_i)$ and we have repeatedly used the fact that $m_i^2=0=\ell_i^2$ on the curve $C$.

Multiplying out expresses $\ch\!\(K_S^3\big/\Lambda\udot(N^{\vir})^\vee\big|_C\)\Td(C)$ as
$$
\Big(\!1+c_1(K_S^3)+\Td_1(C)\!\Big)\!\frac{\prod_i(1-e^{v_it})}{\prod_i(1-e^{w_it})}\!\left(\!1+\sum_i\ell_i\frac{e^{w_it}}{1-e^{w_it}}-\sum_im_i\frac{e^{v_it}}{1-e^{v_it}}\!\right)\!.
$$
Integrating over $C$ gives $\chi\_{e^t}$. Then replacing $e^t$ by $t$ gives $\chi\_t$ as
$$
\frac{\prod_i(1-t^{v_i})}{\prod_i(1-t^{w_i})}
\!\left(\!\deg K_S^3\big|_C+1-g+\sum_i\int_C\!\ell_i\,\frac{t^{w_i}}{1-t^{w_i}}-\sum_i\int_C\!m_i\,\frac{t^{v_i}}{1-t^{v_i}}\!\right)\!.
$$
Substituting \eqref{Li}, \eqref{Mi} and using $\deg K_S|\_C=g-1=\deg\cO_C(C)$ (by adjunction) gives
\begin{align*}
\chi\_t\=&\,\frac{(1-t^{-2})(1-t^{-1})^2}{(1-t^{-1})(1-t)}\cdot\frac{(1-t)^{P_2}}{(1-t^{-2})^{P_2}}\bigg[2g-2+(1-g)\frac{t^{-1}}{1-t^{-1}} \\
+(2g&-2)\frac{t}{1-t}-(2-2g)\frac{t^{-2}}{1-t^{-2}}-(2-2g)\frac{t^{-1}}{1-t^{-1}}-(g-1)\frac{t^{-1}}{1-t^{-1}}\bigg]\\
\=&\left(\frac{-t^2}{1+t}\right)^{\!\!P_2}\frac{2g-2}{t^2}\,.
\end{align*}
Plugging this into \eqref{sofar} gives the contribution of $\cM_2$ to $\mathsf{VW}_{2,K_S,1}(t)$ as
\beq{one}
-t^{2-\frac32P_2}\,\chi\_t\=\frac{(-1)^{P_2}(2-2g)}{\({t^{\frac12}+t^{-\frac12}}\)^{P_2}}\=\frac{(-1)^{P_2}(2-2g)}{[2]_t^{P_2}}\,.
\eeq
\medskip

\noindent\textbf{Horizontal $n=2$ case; preliminaries.}
To calculate on $C^{[2]}$ we fix some notation and collect some results. Let
$$
\xymatrix@=15pt{
\hspace{11mm}\cZ\ \subset\ C^{[2]}\times C \ar[d]<-3.2ex>^q \\
\ C^{[2]}\hspace{1cm}}
$$
denote the universal length-2 subscheme over $C^{[2]}$, with projection $p$ to $C$. Then $\cZ\cong C\times C$ with $p$ the projection to the \emph{first} factor, while the above double cover $q\colon C\times C\to C^{[2]}$ is the quotient by the $\Z/2$ action $\tau$ swapping the factors.

Given a line bundle $L$ on $C$, the induced rank 2 bundle $L^{[2]}=q_*p^*L$ on $C^{[2]}$ is therefore
$$
L^{[2]}\=q_*\(L\boxtimes\cO),
$$
so its pullback by $q^*$ sits inside an exact sequence
\beq{L2}
0\To\tau^*(L\boxtimes\cO)(-\Delta_C)\To q^*L^{[2]}\To L\boxtimes\cO\To0,
\eeq
where $\Delta_C\subset C\times C$ is the diagonal --- the branch divisor of $q$. But $\tau^*(L\boxtimes\cO)=\cO\boxtimes L$, so
\beq{detL2}
q^*\det L^{[2]}\ \cong\ L\boxtimes L(-\Delta_C).
\eeq

The exact sequence
$$
0\To q^*\Omega_{C^{[2]}}\To\Omega_{C\times C}\To\cO_{\Delta_C}(-\Delta_C)\To0,
$$
combined with $\Omega_{C\times C}\cong\Omega_C\boxtimes\cO_C\,\oplus\,\cO_C\boxtimes\Omega_C$ and the exact sequence
$$
0\To\Omega_C\boxtimes\cO_C(-\Delta_C)\To\Omega_C\boxtimes\cO_C\To\cO_{\Delta_C}(-\Delta_C)\To0,
$$
gives an equality in $K$-theory
\beq{OmegaC2}
q^*\Omega_{C^{[2]}}\=\Omega_C\boxtimes\cO_C(-\Delta_C)\ \oplus\ \cO_C\boxtimes\Omega_C\,.
\eeq
In particular,
\beq{KC2}
q^*K_{C^{[2]}}\= K_C\boxtimes K_C(-\Delta_C).
\eeq
As noted in \eqref{ovircalc}, $C^{[2]}\subset S^{[2]}$ is cut out by a transverse section $s^{[2]}$ of $K_S^{[2]}$, so has normal bundle $K_S^{[2]}\big|_{C^{[2]}}=(K_S|\_C)^{[2]}$ and determinant
$$
q^*\!\;\det N_{C^{[2]}/S^{[2]}}\=K_S\big|_C\boxtimes K_S\big|_C(-\Delta_C).
$$
Combining this with \eqref{KC2} gives
\beq{KS2}
q^*\!\left(K_{S^{[2]}}\big|_{C^{[2]}}\right)\ \cong\ K_S\big|_C\boxtimes K_S\big|_C.
\eeq
Substituting (\ref{KS2}, \ref{detL2}) into \eqref{Khalfcalc} we find
$$
q^*\Big(\!K^{\frac12}_{\vir}\big|_{C^{[2]}}\Big)\=K_C\boxtimes K_C\,\t^{4-\frac32P_2}.
$$
Tensoring this by \eqref{ovircalc} shows
\beq{Ohatcalc}
q^*\widehat\cO^{\;\vir}_{\!S^{[2]}}\ \cong\ K_S^3\big|_C\boxtimes K_S^3\big|_C(-\Delta_C)\,\t^{4-\frac32P_2},
\eeq
where, by an abuse of notation, we are considering $\widehat\cO^{\;\vir}_{\!S^{[2]}}$ to be a sheaf on $C^{[2]}$ (really, by \eqref{ovircalc}, it is the K-class of the pushforward of this to $S^{[2]}$).\medskip

\noindent\textbf{Horizontal $n=2$ case; calculation.}
We wish to calculate
\begin{eqnarray}\nonumber
\chi\_{e^t}\!\left(\!\frac{\widehat\cO^{\;\vir}_{\!S^{[2]}}}{\Lambda\udot(N^{\vir})^\vee}\!\right) &=&
\int_{C^{[2]}\ }\frac{\ch\!\(\widehat\cO^{\;\vir}_{\!S^{[2]}}\)}{\ch\!\(\Lambda\udot(N^{\vir})^\vee\)}\Td_{C^{[2]}} \\ \label{integr}
&=& \frac12\int_{C\times C\ }\frac{\ch\!\(q^*\widehat\cO^{\;\vir}_{\!S^{[2]}}\)}{\ch\!\(q^*\Lambda\udot(N^{\vir})^\vee\)}\Td\(q^*T_{C^{[2]}}\).
\end{eqnarray}
Let $k:=c_1(K_C)=-c_1(C)$ and $\chi:=2-2g=-\int_Ck=\Delta_C^2$\;. Putting $s:=e^t$, \eqref{Ohatcalc} gives
\begin{eqnarray} \nonumber
\ch\!\(q^*\widehat\cO^{\;\vir}_{\!S^{[2]}}\) &=& \exp\Big(\frac32k\boxtimes1+\frac32\boxtimes k-[\Delta_C]\Big)(e^t)^{4-\frac32P_2} \\
&=& \Big(1+\frac32k\boxtimes1\Big)\Big(1+\frac32\boxtimes k\Big)\Big(1-[\Delta_C]+\frac12\chi\vol\Big)s^{4-\frac32P_2} \nonumber \\
&=& \Big(1+\frac32k\boxtimes1+\frac32\boxtimes k-[\Delta_C]+\frac14\chi(9\chi+14)\vol\!\Big)s^{4-\frac32P_2}, \label{chvir}
\end{eqnarray}
where $\vol$ is the Poincar\'e dual of a point on $C\times C$. By \eqref{OmegaC2},
\beqa
\ch\!\(q^*T_{C^{[2]}}\) &=& \ch\(T_C\boxtimes\cO_C\)\ch(\cO(\Delta_C))+\ch\(\cO_C\boxtimes T_C\) \\
&=& (1-k\boxtimes1)\(1+[\Delta_C]+\frac12\chi\vol\)+(1-1\boxtimes k) \\
&=& 2-k\boxtimes1-1\boxtimes k+[\Delta_C]+\frac32\chi\vol,
\eeqa
from which we can deduce
\beq{Td}
\Td\!\(q^*T_{C^{[2]}}\)\=1-\frac12k\boxtimes1-\frac12\boxtimes k+\frac12[\Delta_C]+\frac14\chi(2+\chi)\vol.
\eeq
Multiplying \eqref{chvir} and \eqref{Td} makes $\ch\!\(q^*\widehat\cO^{\;\vir}_{\!S^{[2]}}\)
\Td\!\(q^*T_{C^{[2]}}\)$ equal to
\beq{chTd}
s^{4-\frac32P_2}\left(1+k\boxtimes1+1\boxtimes k-\frac12[\Delta_C]+(\chi+\chi^2)\vol\right).
\eeq

For the $K$-theory class of $q^*\(N^{\vir}|_{C^{[2]}}\)^\vee$ we use \eqref{Nvircalc}:
$$
q^*\!\;\(\!\;K_S^{[2]}\)^*\t^{-1}+
q^*(K_S^2)^{[2]}\t+(\t^{-2})^{\oplus P_2}-
q^*\!\;\((K_S^2)^{[2]}\)^*\t^{-2}-
q^*T_{S^{[2]}}\t^{-1}-\t^{\oplus P_2}.
$$
By \eqref{OmegaC2} and several applications of \eqref{L2} this is
\begin{align*}
&\cO_C\boxtimes K_S^{-1}(\Delta_C)\t^{-1}+
K_S^{-1}\boxtimes\cO_C\t^{-1}+
\cO_C\boxtimes K_S^2(-\Delta_C)\t+
K_S^2\boxtimes\cO_C\t \nonumber \\ \nonumber
&-\cO_C\boxtimes K_S^{-2}(\Delta_C)\t^{-2}-
K_S^{-2}\boxtimes\cO_C\t^{-2}
-T_C\boxtimes\cO_C(\Delta_C)\t^{-1}-\cO_C\boxtimes T_C\t^{-1} \end{align*}
\vspace{-8mm}
\beq{mili}
\hspace{2cm}-\cO_C\boxtimes K_S(-\Delta_C)\t^{-1}-K_S\boxtimes\cO_C\t^{-1}+(\t^{-2})^{\oplus P_2}-\t^{\oplus P_2}.
\eeq
Here we have suppressed some restrictions to $C$, which are easily handled using $K_S^2\big|_C\cong K_C$. As in the last section we write this as
$$
q^*(N^{\vir})^\vee\=\sum_iL_i\t^{w_i}-\sum_iM_i\t^{v_i}
$$
where $L_i,\,M_i$ are line bundles with first Chern classes $\ell_i,\,m_i$ respectively. Since $\ell_i^2,\,m_i^2$ needn't be zero\footnote{But we do use that $\l_i^3=0=m_i^3$.} on the surface $C\times C$, \eqref{lm} is modified to
\begin{align*}
\frac1{\ch\(q^*\Lambda\udot(N^{\vir})^\vee\)}\=&\,\frac{\prod_i(1-e^{v_it})}{\prod_i(1-e^{w_it})}\prod_i\!\left[1-\!\left(\!m_i+\frac{m_i^2}2\right)\!\frac{e^{v_it}}{1-e^{v_it}}\right]\times \\
&\qquad \prod_i\!\left[1+\!\left(\!\ell_i+\frac{\ell_i^2}2\right)\!\frac{e^{w_it}}{1-e^{w_it}}+\ell_i^2\frac{e^{2w_it}}{(1-e^{w_it})^2}\right]\!.
\end{align*}
Multiplying by $\frac12$\eqref{chTd} and integrating gives $\chi\_s$ \eqref{integr}. It is the product of
\beq{prefactor}
\frac12s^{4-\frac32P_2}\frac{\prod_i(1-s^{v_i})}{\prod_i(1-s^{w_i})}
\eeq
and the integral over $C\times C$ of
\begin{multline*}
\hspace{-2mm}\left[1+k\boxtimes1+1\boxtimes k-\frac12[\Delta_C]+(\chi+\chi^2)\vol\right]\times \\
\hspace{6mm}\prod_i\!\left[1-\!\left(\!m_i+\frac{m_i^2}2\right)\!\frac{s^{v_i}}{1-s^{v_i}}\right]
\prod_i\!\left[1+\!\left(\!\ell_i+\frac{\ell_i^2}2\right)\!\frac{s^{w_i}}{1-s^{w_i}}+\ell_i^2\frac{s^{2w_i}}{(1-s^{w_i})^2}\right]\!.
\end{multline*}
This integral is
\begin{multline*}
(\chi+\chi^2)-\sum_i\int_{C\times C}\frac{m_i^2}2\frac{s^{v_i}}{1-s^{v_i}}
+\sum_i\int_{C\times C}\frac{\ell_i^2}2\left(\frac{s^{w_i}}{1-s^{w_i}}+\frac{2s^{2w_i}}{(1-s^{w_i})^2}\right)\\ +\sum_{i<j}\int_{C\times C}m_im_j\frac{s^{v_i+v_j}}{(1-s^{v_i})(1-s^{v_j})}
+\sum_{i<j}\int_{C\times C}\ell_i\ell_j\frac{s^{w_i+w_j}}{(1-s^{w_i})(1-s^{w_j})} \\ -\sum_{i,j}\int_{C\times C}m_i\ell_j\frac{s^{v_i+w_j}}{(1-s^{v_i})(1-s^{w_j})} \\ +\int_{C\times C}\Big(k\boxtimes1+1\boxtimes k-\frac12[\Delta_C]\Big)\left(\sum_i\ell_i\frac{s^{w_i}}{1-s^{w_i}}-\sum_im_i\frac{s^{v_i}}{1-s^{v_i}}\right)\!.
\end{multline*}
From \eqref{mili} we read off the $(\ell_i,w_i)$,
$$
\Big(\!-\frac12\boxtimes k+[\Delta_C],-1\Big),\ \Big(\!-\frac12k\boxtimes1,-1\Big),\ 
\(1\boxtimes k-[\Delta_C],1\),\ (k\boxtimes1,1)
$$
and $P_2$ copies of $(0,-2)$. Similarly the $(m_i,v_i)$ are
\begin{multline*}
(-1\boxtimes k+[\Delta_C],-2),\ (-k\boxtimes1,-2),\ 
(-k\boxtimes1+[\Delta_C],-1),\ (-1\boxtimes k,-1), \\
\Big(\frac12\boxtimes k-[\Delta_C],-1\Big),\ \Big(\frac12k\boxtimes1,-1\Big)
\end{multline*}
and $P_2$ copies of $(0,1)$. Substituting these into the integral gives
$$
\frac1{(1-s^2)^2}\Big[s^2\chi^2+(2s^3+5s^2+2s)\chi\Big].
$$
The prefactor \eqref{prefactor} is
\beqa
\frac12s^{4-\frac32P_2}\frac{\prod_i(1-s^{v_i})}{\prod_i(1-s^{w_i})} &=&
\frac12s^{4-\frac32P_2}\frac{(1-s^{-2})^2(1-s^{-1})^4(1-s)^{P_2}}{(1-s^{-1})^2(1-s)^2(1-s^{-2})^{P_2}} \\
&=& (-1)^{P_2}\frac{s^{-2}(1-s^2)^2}{2(s^{\frac12}+s^{-\frac12})^{P_2}}
\eeqa
Therefore, replacing $s$ by $t$, we find the contribution of $S^{[2]}$ to the refined Vafa-Witten invariant is
\beq{n2vert}
(-1)^{P_2\,}\frac{\chi^2+(2t+5+2t^{-1})\chi}{2[2]_t^{P_2}}\,.
\eeq
At $t=1$ this gives $(-2)^{-P_2}(1-g)(11-2g)$, as in \cite[Equation 8.24]{TT1}. \medskip

\noindent\textbf{Vertical $n=2$ case.} When $n=2$ there is another component of the $T$-fixed locus, given by taking $i=1$ in \eqref{nested}. This gives a copy of $S$, where $x\in S$ corresponds to the sheaf
$$
\(\rho^*I_x\)\otimes\cO_{2S}\ \ \mathrm{on\ }X.
$$
In \cite[Section 8.7]{TT1} it is shown that this $T$-fixed moduli space has vanishing obstruction sheaf, so that
$$
\cO^{\vir}_S\=\cO_S,
$$
and virtual normal bundle
$$
N^{\vir}\=T_S\!\otimes\!K_S^{-1}\;\t^{-1}\,\oplus\,
H^0(K_S^2)\;\t^2-\Big[\Omega_S\;\t\,\oplus\,T_S\!\otimes\!K_S^2\;\t^2\,\oplus\,H^0(K_S^2)^*\;\t^{-1}\Big].
$$
In particular
$$
K_{\vir}\big|_S\=K_S\otimes\!\(\!\;K_S\otimes K_S^2\;\t^2\)\!\otimes\!\(\t^{-2P_2}\)\!\otimes\!\(\!\;K_S\;\t^2\)\!\otimes\!\(\!\;K_S^{-1}\otimes K_S^4\;\t^4\)\!\otimes\!\(\t^{-P_2}\)
$$
whose square root \eqref{sqroot} is
$$
K_{\vir}^{\frac12}\big|_S\=K_S^4\;\t^{4-\frac32P_2}.
$$
So the contribution of this component to the refined Vafa-Witten invariant is, by \eqref{chiloc},
\begin{align*}
t^{4-\frac32P_2}&\chi\_t\!\left(\!S,\frac{K_S^4}{\Lambda\udot(N^{\vir})^\vee}\right) \\
\=&t^{4-\frac32P_2}\frac{(1-t)^{P_2}}{(1-t^{-2})^{P_2}}\,\chi\_t\!\left(\!S,\frac{\Lambda\udot(T_S\;\t^{-1})\otimes\Lambda\udot(\Omega_S\otimes K_S^{-2}\;\t^{-2})\otimes K_S^4}{\Lambda\udot(\Omega_S\otimes K_S\;\t)}\right) \\
\=&\frac{(-1)^{P_2}t^{-2}}{(t^{-\frac12}+t^{\frac12})^{P_2}}\,\chi\_t\!\left(\!S,\frac{\Lambda\udot(T_S\;\t^{-1})\otimes K_S\;\t^2\otimes\Lambda\udot(\Omega_S\otimes K_S^{-2}\;\t^{-2})\otimes K_S^3\;\t^4}{\Lambda\udot(\Omega_S\otimes K_S\;\t)}\!\right) \\
\=&\frac{(-1)^{P_2}t^{-2}}{[2]_t^{P_2}}\,\chi\_t\!\left(\!S,\frac{\Lambda\udot(\Omega_S\;\t)\otimes\Lambda\udot(\Omega_S\otimes K_S\;\t^2)}{\Lambda\udot(\Omega_S\otimes K_S\;\t)}\right)\!.
\end{align*}
Using the canonical section $s$ with zero locus $C$ we see that in $K$-theory,
\beqa
\Omega_S\otimes K_S\;\t-\Omega_S\;\t &=& \(\Omega_S\otimes K_S\)\big|_C\;\t\=K_S\otimes\(\Omega_C\oplus K_S^{-1}\big|_C\)\t \\
&=& K_S^3\big|_C\;\t\ \oplus\ \cO_C\;\t\=K_S^3\;\t-K_S^2\;\t
+\cO_S\;\t-K_S^{-1}\;\t.
\eeqa
Therefore
$$
\frac{\Lambda\udot(\Omega_S\;\t)}{\Lambda\udot(\Omega_S\otimes K_S\;\t)}\=
\frac{\(\cO_S-K_S^2\t\)\(\cO_S-K_S^{-1}\t\)}{\(\cO_S-K_S^3\t\)\(\cO_S-\t\)}\,,
$$
while
\beqa
\Lambda\udot(\Omega_S\otimes K_S\;\t^2) &=&
\Lambda\udot(\Omega_S\;\t^2)\Lambda\udot\((\Omega_S\otimes K_S)\big|_C\;\t^2\) \\ &=&
\Lambda\udot(\Omega_S\;\t^2)\frac{\(\cO_S-K_S^3\t^2\)\(\cO_S-\t^2\)}{\(\cO_S-K_S^2\t^2\)\(\cO_S-K_S^{-1}\t^2\)}\,.
\eeqa
Putting it all together gives
\beq{together}
\frac{t^{-2}(1+t)}{(-[2]\_t)^{P_2}}\ \chi\_t\!\left(\!
\frac{\(\cO_S-K_S^2\t\)\(\cO_S-K_S^{-1}\t\)\Lambda\udot(\Omega_S\;\t^2)\(\cO_S-K_S^3\t^2\)}{\(\cO_S-K_S^3\t\)\(\cO_S-K_S^2\t^2\)\(\cO_S-K_S^{-1}\t^2\)}\right)\!.
\eeq
By Riemann-Roch $\chi\_{e^t}$ is
$$
\int_S\frac{\ch\big[\(\cO_S-K_S^2\t\)\(\cO_S-K_S^{-1}\t\)\(\cO_S-K_S^3\t^2\)\big]}{\ch\big[\(\cO_S-K_S^3\t\)\(\cO_S-K_S^2\t^2\)\(\cO_S-K_S^{-1}\t^2\)\big]}\ch\!\(\Lambda\udot(\Omega_S\;\t^2)\)\Td_S.
$$
Let $\kappa:=c_1(K_S)=-c_1(S)$ and $c_2:=c_2(S)$, this is
$$
\int_S\frac{\(1-e^{2\kappa}e^t\)\(1-e^{-\kappa}e^t\)\(1-e^{3\kappa}e^{2t}\)}{\(1-e^{2\kappa}e^{2t}\)\(1-e^{-\kappa}e^{2t}\)\(1-e^{3\kappa}e^t\)}\(1-\ch(\Omega_S)e^{2t}+e^{\kappa}e^{4t}\)\!\Td_S.
$$
By \eqref{together}, therefore, the vertical contribution is
\beq{form}
\frac{t^{-2}(1+t)}{(-[2]\_t)^{P_2}}\int_S\frac{\(1-e^{2\kappa}t\)\(1-e^{-\kappa}t\)\(1-e^{3\kappa}t^2\)}{\(1-e^{2\kappa}t^2\)\(1-e^{-\kappa}t^2\)\(1-e^{3\kappa}t\)}\(1-\ch(\Omega_S)t^2+e^{\kappa}t^4\)\!\Td_S.
\eeq
When $\alpha^3=0$ we compute
\beqa
\frac{1-e^{\alpha}t}{1-e^{\alpha}t^2} &=& \frac{(1-t)\Big(1-\frac t{1-t}\alpha-\frac t{2(1-t)}\alpha^2\Big)}{(1-t^2)\left(1-\frac{t^2}{1-t^2}\alpha-\frac{t^2}{2(1-t^2)}\alpha^2\right)}\\ &=& \frac1{1+t}\left(1-\frac t{1-t}\alpha-\frac t{2(1-t)}\alpha^2\right) \\ && \hspace{1cm}\times\left(1+\frac{t^2}{1-t^2}\alpha+\frac{t^2}{2(1-t^2)}\alpha^2+\frac{t^4}{(1-t^2)^2}\alpha^2\right) \\
&=& \frac1{1+t}\left(1-\frac t{1-t^2}\alpha-\frac{t(1+t^2)}{2(1-t^2)^2}\alpha^2\right)\!.
\eeqa
Similarly
$$
\frac{1-e^{\alpha}t^2}{1-e^{\alpha}t}\=(1+t)\left(1+\frac t{1-t^2}\alpha+\frac{t}{2(1-t)^2}\alpha^2\right)\!.
$$
Multiplying these gives
$$
\frac{\(1-e^{2\kappa}t\)\(1-e^{-\kappa}t\)\(1-e^{3\kappa}t^2\)}{\(1-e^{2\kappa}t^2\)\(1-e^{-\kappa}t^2\)\(1-e^{3\kappa}t\)}\=
\frac1{1+t}\!\left(\!1+\frac{2t}{1-t^2}\kappa+\frac{2t}{(1-t)^2}\kappa^2\!\right)\!.
$$
Now $\(1-\ch(\Omega_S)t^2+e^{\kappa}t^4\)\!\Td_S$ is
\begin{multline*}
\Big(1-2t^2-t^2\kappa-\frac{t^2}2(\kappa^2-2c_2)+t^4\Big(1+\kappa+\frac{\kappa^2}2\Big)\Big)\Big(1-\frac{\kappa}2+\frac1{12}(\kappa^2+c_2)\Big) \\
\=(1-t^2)^2-\frac12(1-t^4)\kappa+\frac1{12}(1-t^2)(1+11t^2)\kappa^2+\frac1{12}(1-t^2)^2c_2.
\end{multline*}
Plugging all this into \eqref{form} gives $t^{-2}(-[2]\_t)^{-P_2}$ times by
\begin{multline*}
\int_S\left(\!1+\frac{2t}{1-t^2}\kappa+\frac{2t}{(1-t)^2}\kappa^2\!\right)\!\left(\!(1-t^2)-\frac{1+t^2}2\kappa+\frac{1+11t^2}{12}\kappa^2+\frac{1-t^2}{12}c_2\!\right) \\
\=\frac{1+12t+46t^2+12t^3+t^4}{12}(g-1)+\frac{1+10t^2+t^4}{12}c_2(S).
\end{multline*}
So the vertical contribution to the refined Vafa-Witten invariant is
$$
\frac1{12(-[2]\_t)^{P_2}}\Big[
\(t^{-2}+12t^{-1}+46+12t+t^2\)(g-1)+(t^{-2}+10+t^2)c_2(S)\Big].
$$
Setting $t=1$ gives $(-2)^{-P_2}\(6(g-1)+c_2)$ in agreement with \cite[Equation 8.28]{TT1}.
\medskip

\noindent\textbf{Total.}
Adding this to the horizontal contribution \eqref{n2vert} at $n=2$ gives a total
\begin{multline*}
\frac{(-1)^{P_2}}{12[2]_t^{P_2}}\Big[24(g-1)^2+
\(t^{-2}-12t^{-1}-14-12t+t^2\)(g-1)\\ +(t^{-2}+10+t^2)c_2(S)\Big]\!.
\end{multline*}
Combining this with \eqref{zero} and \eqref{one} gives the generating series
\begin{multline*}
\frac{(-1)^{P_2}}{[2]_t^{P_2}}\bigg[1\ +\ (2-2g)q\ + \\
\Big(\!24(g-1)^2+\(t^{-2}-12t^{-1}-14-12t+t^2\)(g-1)+(t^{-2}+10+t^2)c_2(S)\!\Big)\frac{q^2}{12}\bigg]
\end{multline*}

Many thanks to Martijn Kool who pointed out this formula perfectly matches the first few terms of calculations of G\"ottsche-Kool \cite{GK1} after applying the most naive modularity transformation.
Specialising to $t=1$ gives
$$
(-2)^{-P_2}\Big[1+(2g-2)q+\Big((g-1)(2g-5)+c_2(S)\Big)q^2\Big]+O(q^3)
$$
as in \cite[Equation 8.39]{TT1}, or indeed the second term of the first line of \cite[Equation 5.38]{VW}.

For higher $c_2$ we need a more systematic way to compute. Laarakker \cite{La1} combines the degeneracy locus description of monopole branches in \cite{GT1} and Section \ref{degensec} of this paper with cobordism arguments to prove universality results. This also translates computations to calculations on toric surfaces, which can be done by localisation and computer for $c_2$ not too large.

\section{Refined Vafa-Witten invariants: semistable case}
\label{sscase}
As before we fix a polarised surface $(S,\cO_S(1))$. Pulling back gives a polarisation $\cO_X(1)$ on the local Calabi-Yau threefold $X=K_S$.
We define the \emph{charge} of a compactly supported coherent sheaf $\cE$ on $X$ to be the total Chern class
\beq{alph}
\alpha\ =\ (r,c_1,c_2)\ \in\ H^{\ev}(S)
\eeq
of the pushdown $E=\rho_*\;\cE$ on $S$. Given $n\gg0$ and $L\in\Pic_{c_1}(S)$, an $SU(r)$-Joyce-Song pair $(\cE,s)$ consists of
\begin{itemize}
\item a compactly supported coherent sheaf $\cE$ of charge $\alpha$ on $X$, with centre-of-mass zero on each fibre of $\rho\colon X\to S$ and $\det\rho_*\;\cE\cong L$, and
\item a \emph{nonzero} section $s\in H^0(\cE(n))$.
\end{itemize}
Equivalently, it is a triple $(E,\phi,s)$ on $S$ with $\phi\in\Hom(E,E\otimes K_S)\_0,\ \det E\cong L$ and $s\in H^0(E(n))$.
The Joyce-Song pair $(\cE,s)$ is \emph{stable} if and only if
\begin{itemize}
\item $\cE$ is Gieseker semistable with respect to $\cO_X(1)$, and
\item if $\cF\subset\cE$ is a proper subsheaf which destabilises $\cE$, then $s$ does \emph{not} factor through $\cF(n)\subset\cE(n)$.
\end{itemize}
For fixed $\alpha$ we may choose $n\gg0$ such that $H^{\ge1}(\cE(n))=0$ for all Joyce-Song stable pairs $(\cE,s)$.
There is no notion of semistability; there is a quasi-projective moduli scheme $\cP^\perp_{\alpha,n}$ of stable Joyce-Song pairs whose $T=\C^*$-fixed locus is already compact.

Most importantly, $\cP^\perp_{\alpha,n}$ can be shown to be a moduli space of complexes $I\udot:=\{\cO_X(-n)\to\cE\}$ with a symmetric perfect obstruction theory governed by $R\Hom(I\udot,I\udot)\_\perp$. As a result it inherits a virtual structure sheaf and the virtual canonical bundle admits a canonical square root \eqref{sqroot} on the $T$-fixed locus. Thus by \eqref{tloc} we get a refined pairs invariant
$$
P^\perp_{\alpha}(n,t)\ :=\ \chi\_t\Big(\cP^\perp_{\alpha,n},\widehat\cO^{\;\vir}_{\!\cP^\perp_{\alpha,n}}\Big)\ \in\ \Q(t^{1/2}).
$$
Using the quantum integers defined in \eqref{qq},
$$
[n]_q\,:=\ q^{-(n-1)/2}+q^{-(n-3)/2}+\cdots+q^{(n-3)/2}+q^{(n-1)/2}\=\frac{q^{n/2}-q^{-n/2}}{q^{1/2}-q^{-1/2}}
$$
(which specialise to $n$ at $q=1$) we can state the refined version of \cite[Conjecture 6.5]{TT2}.

\begin{conj} \label{pechconj}
Suppose $\cO_S(1)$ is generic for charge $\alpha$ in the sense of \cite[Equation 2.4]{TT2}.\footnote{This ensures the charges of the semistable sheaves $E_i$ in the splitting \eqref{wtspace} are all proportional to $\alpha$. For non-generic $\cO_S(1)$ there is a more complicated version of \eqref{desk}.} If $H^{0,1}(S)=0=H^{0,2}(S)$ there exist $\VW_{\alpha_i}(t)\in\Q(t^{1/2})$ such that
\beq{desk}
P^\perp_{\alpha}(n,t)\ =\ \mathop{\sum_{\ell\ge 1,\,(\alpha_i=\delta_i\alpha)_{i=1}^\ell:}}_{\delta_i>0,\ \sum_{i=1}^\ell\delta_i=1}
\frac{(-1)^\ell}{\ell!}\prod_{i=1}^\ell(-1)^{\chi(\alpha_i(n))} \big[\chi(\alpha_i(n))\big]_t\VW_{\alpha_i}(t) \hspace{-2mm}
\eeq
for $n\gg0$.
When either of $H^{0,1}(S)$ or $H^{0,2}(S)$ is nonzero we take only the first term in the sum:
\beq{desk2}
P^\perp_\alpha(n,t)\ =\ (-1)^{\chi(\alpha(n))-1}\big[\chi(\alpha(n))\big]\_t\VW_\alpha(t).
\eeq
\end{conj}\medskip

The first justification for this Conjecture is that it specialises at \mbox{$t=1$} to Conjecture 6.5 of \cite{TT2}, which is proved in many cases \cite{MT1,TT2}. Therefore, when the Conjecture holds, the resulting $\VW_\alpha(t)$ specialise at $t=1$ to the numerical Vafa-Witten invariants of \cite{TT2}.

As a second sanity check, we show it is true --- and reproduces the earlier Definition \ref{defref} of refined Vafa-Witten invariants --- when there are no strictly semistable sheaves.

\begin{prop} \label{s-ss} If all semistable sheaves in $\cN_\alpha^\perp$
are stable then Conjecture \ref{pechconj} is true with the coefficients $\VW_\alpha(t)$ given by \eqref{eqref}.
\end{prop}

\begin{proof}
We adapt the proof of the corresponding result for numerical Vafa-Witten invariants in \cite[Proposition 6.8]{TT2}.

We proceed by induction on the rank $r$ of $\alpha=(r,c_1,c_2)$.
We first claim that if there are no strictly semistables in class $\alpha$ then only the first term contributes to the sum \eqref{long}. Indeed, if there was a nonzero contribution indexed by $\alpha_1,\ldots,\alpha_\ell$ with $\ell>1$ then the nonvanishing of the coefficients $\VW_{\alpha_i}(t)$ (which equal the refined Vafa-Witten invariants \eqref{eqref} by the induction hypothesis) would imply that the moduli spaces $\cN^\perp_{\alpha_i}$ are nonempty. Picking an element $\cE^i$ of each defines a strictly semistable $\cE:=\cE^1\oplus\cdots\oplus\cE^\ell$ of $\cN_\alpha^\perp$, a contradiction.

We use the smooth $\PP^{\chi(\alpha(n))-1}$-bundle
$$
\cP^\perp_{\alpha,n}\=\PP(\pi_*\;\curly E(n))\ \rt{p}\ \cN^\perp_{\alpha},
$$
where $\curly E$ is the (possibly twisted) universal sheaf on $\cN_\alpha^\perp\times X$ and $\pi\colon\cN^\perp_\alpha\times X\to\cN_\alpha$. There is a corresponding relationship between the deformation-obstruction theories of $(\cN_\alpha^\perp)^T$ and $(\cP_{\alpha,n}^\perp)^T$ worked out in \cite[Equations 6.12--6.14]{TT2}. In particular \cite[Equation 6.13]{TT2} implies that the virtual structure sheaves of their $T$-fixed loci satisfy
$$
\cO^{\;\vir}_{\!(\cP^\perp_{\alpha,n})^T}\=p^*\;\cO^{\;\vir}_{\!(\cN^\perp_{\alpha})^T}\otimes\Lambda\udot\big(T_p\otimes\t^{-1}\big)^{\mathrm{fix}}\,,
$$
where $T_p$ is the relative tangent bundle of $p$. And by \cite[Equation 6.14]{TT2} their dual virtual normal bundles are related by
$$
\(N^{\vir}_{(\cP^\perp_\alpha)^T}\)^\vee\ =\ p^*\(N^{\vir}_{(\cN^\perp_\alpha)^T}\)^\vee\ +\ N^*_{(\cP^\perp_\alpha)^T/\cP^\perp_\alpha} -\ \big(T_p\otimes\t^{-1}\big)^{\mathrm{mov}}.
$$
Taking the determinant of \cite[Equation 6.12]{TT2} gives, by \eqref{sqroot},
$$
K^{\frac12}_{\cP^\perp_{\alpha,n}}\=p^*K^{\frac12}_{\cN^\perp_{\alpha}}\otimes\omega_p\;\t^{\frac12\dim p}.
$$
Putting it all together we have
\beqa
\frac{\widehat\cO^{\;\vir}_{\!(\cP^\perp_{\alpha,n})^T}}
{\Lambda\udot\(N^{\vir}_{(\cP^\perp_{\alpha,n})^T})^\vee} &=&
p^*\frac{\widehat\cO^{\;\vir}_{\!(\cN^\perp_{\alpha})^T}}
{\Lambda\udot\(N^{\vir}_{(\cN^\perp_{\alpha})^T})^\vee}\otimes\frac{\Lambda\udot\(T_p\otimes\t^{-1}\)}{\Lambda\udot N^*_{(\cP^\perp_\alpha)^T/\cP^\perp_\alpha}}\otimes\omega_p\t^{\frac12\dim p} \\
&=& p^*\frac{\widehat\cO^{\;\vir}_{\!(\cN^\perp_{\alpha})^T}}
{\Lambda\udot\(N^{\vir}_{(\cN^\perp_{\alpha})^T})^\vee}\otimes\frac{\Lambda\udot\(\Omega_p\otimes\t\)}{\Lambda\udot N^*_{(\cP^\perp_\alpha)^T/\cP^\perp_\alpha}}\t^{-\frac12\dim p}(-1)^{\dim p}
\eeqa
by the identity \eqref{indenti} applied to $T_p\otimes\t^{-1}$. To take $\chi\_t$ we first push down the restriction of $p$ to $(\cP_{\alpha,n}^\perp)^T\to(\cN^\perp_\alpha)^T$. This is a smooth bundle; on each fibre we get
\beq{fibe}
\chi\_t\!\left(\PP\;^T,\frac{\Lambda\udot\(\Omega_{\PP}\otimes\t\)}{\Lambda\udot N^*_{\PP^T/\PP}}\right),
\eeq
where $\PP=\PP^{\chi(\alpha(n))-1}$ is acted on by $T$ with fixed locus $\PP\;^T$. We recognise \eqref{fibe} as the computation of
\beqa
\chi\_t\(\PP,\Lambda\udot\(\Omega_{\PP}\otimes\t\)\) &=&
\sum_{i,j=0}^{\chi(\alpha(n))-1}(-1)^{i+j}\chi\_t\(H^i(\Omega^j_{\PP}))\;t^j \\
&=& 1+t+t^2+\cdots+t^{\chi(\alpha(n))-1}\=t^{\frac12\dim p}[\chi(\alpha(n))]\_t
\eeqa
by localisation to the fixed locus $\PP\;^T$. (Here we have used the fact that $T$ acts trivially on $H^i(\Omega^j_{\PP}))$ since the latter is topological.) Moreover, there is no twisting as we move over the base --- the fibrewise cohomology groups of a $\PP^{\chi(\alpha(n))-1}$ bundle are canonically trivialised by powers of the hyperplane class.\footnote{If the universal sheaf $\curly E$ is twisted by a nonzero Brauer class then $\cP^\perp_{\alpha,n}=\PP(p_*\curly E(n))$ is not the projectivisation of an untwisted bundle, so the hyperplane class does not lift to $H^2(\cP^\perp_{\alpha,n})$. But its fibrewise class in $H^0(\cN^\perp_{\alpha,n},R^1p_*\;\Omega_p)$ is well-defined and is all we use.} So the upshot is
\beq{quant}
Rp_*\,\frac{\widehat\cO^{\;\vir}_{\!(\cP^\perp_{\alpha,n})^T}}
{\Lambda\udot\(N^{\vir}_{(\cP^\perp_{\alpha,n})^T})^\vee}\=
\frac{\widehat\cO^{\;\vir}_{\!(\cN^\perp_{\alpha})^T}}
{\Lambda\udot\(N^{\vir}_{(\cN^\perp_{\alpha})^T})^\vee}\cdot(-1)^{\chi(\alpha(n))-1}\, [\chi(\alpha(n))]\_t
\eeq
in $K$-theory. Taking $\chi\_t$ gives
\[
P^\perp_{\alpha}(n,t)\=(-1)^{\chi(\alpha(n))-1}\big[\chi(\alpha(n))\big]\_t\;\VW_\alpha(t).\qedhere
\]
\end{proof}

\subsection{\for{toc}{$\deg K_S<0$: refined DT invariants}\except{toc}{\bf deg $\mathbf{K_S<0}$: refined DT invariants}}
\label{K<0}
When $\deg K_S<0$ the moduli space $\cP^\perp_{\alpha,n}$ of Joyce-Song pairs on $X$ is smooth \cite[Section 6.1]{TT2} and consists entirely of pairs pushed forward (scheme theoretically) from $S$. The obstruction bundle is $T^*_{\cP^\perp_{\alpha,n}}\otimes\t$. It follows from Proposition \ref{tstar-1} that the refined pairs invariant is
\begin{eqnarray} \nonumber
P^\perp_{\alpha,n}(t) &=& (-1)^{\dim\cP^\perp_{\alpha,n}}t^{-\dim\cP^\perp_{\alpha,n}/2}\chi\_{-t}(\cP^\perp_{\alpha,n}) \\ \label{refP}
&=& (-1)^{\chi(\alpha(n))-\chi\_S(\alpha,\alpha)}
t^{-\frac12\chi(\alpha(n))+\frac12\chi\_S(\alpha,\alpha)}\chi\_{-t}(\cP^\perp_{\alpha,n}).
\end{eqnarray}
The Vafa-Witten obstruction theory on $\cP^\perp_{\alpha,n}$ is the DT obstruction theory of \cite{JS} since $H^1(\cO_S)=0=H^2(\cO_S)$. So we can expect the refined Vafa-Witten invariants to be closely related to refined DT invariants,\footnote{Since the existence of orientation data \emph{compatible with the Hall algebra product} is still an open problem in general, the development of refined DT theory has stalled. In our situation all our moduli stacks are $(-1)$-shifted cotangent bundles of smooth stacks, which makes things much simpler.} and this is what we will find.

We use Joyce's Ringel-Hall algebra for $\mathrm{Coh}(S)$ constructed in \cite{Jo2}. Joyce starts with the $\Q$-vector space on generators given by (isomorphism classes of) representable morphisms of stacks from algebraic stacks of finite type over $\C$ with affine stabilisers to the stack of objects of Coh$(S)$. He then quotients out by the scissor relations for closed substacks, and makes the result into a ring with his Hall algebra product $*$ on stack functions.

At the level of individual objects, the product $1_E*1_F$ of (the indicator functions of) $E$ and $F$ is the stack of all extensions between them,
\beq{extstack}
\frac{\Ext^1(F,E)}{\Aut(E)\times\Aut(F)\times\Hom(F,E)}\,,
\eeq
with $e\in\Ext^1(F,E)$ mapping to the corresponding extension of $F$ by $E$.
More generally $*$ is defined via the stack $\mathfrak{Ext}$ of all short exact sequences
\beq{exten}
0\To E_1\To E\To E_2\To0
\eeq
in Coh$(S)$, with its morphisms $\pi_1,\pi,\pi_2\colon\mathfrak{Ext}\to$ Coh$(S)$ taking the extension to $E_1,\,E,\,E_2$ respectively. This defines the universal case, which is the Hall algebra product of Coh$(S)$ with itself:
$$
1_{\mathrm{Coh}(S)}*1_{\mathrm{Coh}(S)}\ =\ \Big(\mathfrak{Ext}\rt\pi\mathrm{Coh}(S)\Big).
$$
Other products are defined by fibre product with this: given two stack functions $U,V\to$ Coh$(S)$ we define $U*V\to$ Coh$(S)$ by the Cartesian square
\beq{Cart}
\xymatrix@C=15pt{
U*V \ar[r]\ar[d]& \mathfrak{Ext} \ar[r]^(.42)\pi\ar[d]_{\pi_1\!}^{\!\times\pi_2}& \mathrm{Coh}_c(X) \\
U\times V \ar[r]& \mathrm{Coh}(S)\times\mathrm{Coh}(S)\,.\!\!\!}
\eeq

We are interested in the elements
$$
1_{\cM^{ss}_\alpha}\colon\,\cM^{ss}_\alpha\ \Into\ \mathrm{Coh}(S),
$$
where $\cM^{ss}_\alpha$ is the \emph{stack} of Gieseker semistable sheaves of class $\alpha$ on $(S,\cO_S(1))$, and $1_{\cM^{ss}_\alpha}$ is its inclusion into the stack of all sheaves on $S$.
To handle the stabilisers of strictly semistable sheaves, Joyce replaces these indicator stack functions by their ``logarithm": the following (finite!) sum:
\beq{epsi}
\epsilon_\alpha\ :=\ \mathop{\sum_{\ell\ge 1,\,(\alpha_i)_{i=1}^\ell:\,p_{\alpha_i}=\,p_{\alpha}\ \forall i}}_{\mathrm{and}\ \sum_{i=1}^\ell\alpha_i=\alpha}\frac{(-1)^\ell}\ell\ 1_{\cM^{ss}_{\alpha_1}}*\cdots*1_{\cM^{ss}_{\alpha_\ell}}\,.
\eeq
Here $p_\alpha$ denotes the reduced Hilbert polynomial of sheaves of class $\alpha$.

A deep result of Joyce \cite[Theorem 8.7]{Jo3} is that the logarithm \eqref{epsi} lies in the set of \emph{virtually indecomposable stack functions with algebra stabilisers}, $$
\epsilon_\alpha\ \in\ \mathrm{SF}_{\mathrm{al}}^{\mathrm{ind}}(\mathrm{Coh}(S)).
$$
By \cite[Proposition 3.4]{JS} this means it can be written as a $\Q$-linear combination of morphisms from stacks of the form
\beq{BC*}
\text{(scheme)}\times(\Spec\C)/\C^*.
\eeq

Now all moduli stacks of semistable torsion free sheaves on $S$ of class $\alpha$ are smooth of dimension $-\chi\_S(\alpha,\alpha)$, since any obstruction space $\Ext^2(E,E)$ is Serre dual to $\Hom(E,E\otimes K_S)$, which vanishes by the semistability of $E$ and the negativity of $\deg K_S$. Therefore we do not have to worry about vanishing cycles or orientation data; we can make a naive definition of the Joycian refined DT invariant\footnote{There is a further refinement given by taking the two variable Hodge-Deligne polynomial $E(u,v)$ of the stack $\epsilon_\alpha$. Here we take its specialisation $\chi\_{-t}=E(t,1)$.} by taking the normalised Hirzebruch $\chi\_{-t}$-genus of the ($\Q$-linear combination of) stacks $\epsilon_\alpha$:
\beq{JDT}
\mathsf{J}_\alpha(t)\ :=\ (-1)^{1-\chi\_S(\alpha,\alpha)}t^{\frac12(\chi\_S(\alpha,\alpha)-1)}(t-1)\chi\_{-t}(\epsilon_\alpha)\ \in\ \Q\(t^{\pm\frac12}\).
\eeq
The factor $(t-1)$ is there to cancel the $\C^*$ stabilisers in \eqref{BC*}. Joyce's result that $\epsilon_\alpha$ is a virtual indecomposable therefore means that $\mathsf{J}_\alpha(t)$ has a finite limit as $t\to1$; this limit is the numerical DT invariant.

We can use this to prove a refined version of the Joyce-Song identity \cite[Theorem 5.27]{JS}, and hence our Conjecture \ref{pechconj}, when $\deg K_S<0$.

\begin{thm}\label{VW=J} If $\deg K_S<0$ and $\cO_S(1)$ is generic then
$$
P^\perp_{\alpha}(n,t)\ =\ \mathop{\sum_{\ell\ge 1,\,(\alpha_i=\delta_i\alpha)_{i=1}^\ell:}}_{\delta_i>0,\ \sum_{i=1}^\ell\delta_i=1}
\frac{(-1)^\ell}{\ell!}\prod_{i=1}^\ell(-1)^{\chi(\alpha_i(n))} \big[\chi(\alpha_i(n))\big]_{t\,}\mathsf{J}_{\alpha_i}(t).
$$
Since $H^1(\cO_S)=0=H^2(\cO_S)$ this shows that Conjecture \ref{pechconj} holds, with the refined Vafa-Witten invariants equalling the refined DT invariants
$$
\VW_{\alpha}(t)\=\mathsf J_\alpha(t).
$$
\end{thm}

Here, as before \eqref{qq}, $[\chi(\alpha(n))]_t$ is the symmetrised Poincar\'e polynomial (or Hirzebruch $\chi\_{-t}$ genus) of $\PP(H^0(E(n)))$ for any sheaf $E$ of charge $\alpha$. This refines the Euler characteristic $\chi(\alpha(n))$ of $\PP(H^0(\cE(n)))$ that appears in \cite[Theorem 5.27]{JS}.

Because the Vafa-Witten perfect obstruction theory coincides with the DT obstruction theory when $H^1(\cO_S)=0=H^2(\cO_S)$, this theorem is an instance of Maulik's result \cite{Ma} that, for some local Calabi-Yaus, the refinement of Joyce/Kontsevich-Soibelman coincides with the refinement of Nekrasov-Okounkov.

\begin{proof}
We follow \cite[Chapter 13]{JS}, fixing $n\gg0$ and use the same auxiliary categories $\cB_{p_\alpha}$ (whose objects are Gieseker semistable sheaves $E$ on $S$ with reduced Hilbert polynomial a multiple of $p_\alpha$, plus a vector space $V$ and a linear map $V\to H^0(E(n))$) with $K$-theory classes in $K(\mathrm{Coh}(S))\oplus\Z$. We use the Euler form $\bar\chi\((\alpha,d),(\beta,e)\):=\chi\_S(\alpha,\beta)-d\chi(\alpha(n))-e\chi(\beta(n))+de\chi(\cO_S)$. Gieseker stability on Coh$(S)$ induces stability conditions on $\cB_{p_\alpha}$. Wall crossing between them gives the following identity of stack functions \cite[Proposition 13.10]{JS}\footnote{We have used the genericity of $\cO_S(1)$ to simplify the sum to one over only charges proportional to $\alpha$.}
\beq{eee}
\epsilon^{(\alpha,1)}\=\mathop{\sum_{\ell\ge 1,\,(\alpha_i=\delta_i\alpha)_{i=1}^\ell:}}_{\delta_i>0,\ \sum_{i=1}^\ell\delta_i=1}\frac{(-1)^\ell}{\ell!}\big[\big[\cdots\big[\big[\epsilon^{(0,1)},\epsilon_{\alpha_1}\big],\epsilon_{\alpha_2}\big],\cdots\big],\epsilon_{\alpha_\ell}\big].
\eeq
Here the Lie bracket is with respect to the Hall algebra product on stack functions, and $\epsilon_\alpha$ is the stack function \eqref{epsi} mapping to the stack of semistable sheaves in Coh$(S)$ thought of as semistable objects in $\cB_{p_\alpha}$. Then $\epsilon^{(\alpha,1)}$ is defined in a similar way from objects in $\cB_{p_\alpha}$ (with the $K$-theory class of $\cO(-n)\to E$, for $E$ of class $\alpha$) which are semistable with respect to Joyce-Song's stability condition $\tilde\tau$. As in \cite[13.5]{JS}, it is the stack $\cP^\perp_{\alpha,n}/\C^*$. \medskip

Set $\alpha_{<k}:=\alpha_1+\ldots+\alpha_{k-1}$ and $\alpha_{\le k}:=\alpha_1+\ldots+\alpha_k$. We multiply out the Lie brackets in \eqref{eee}, starting with the innermost one. We claim that by induction, at the $k$\;th stage, we get a bracket
\beq{A*E}
A^{(\alpha_{<k},1)}*\epsilon_{\alpha_k}-\epsilon_{\alpha_k}*A^{(\alpha_{<k},1)},
\eeq
where $A^{(\alpha_{<k},1)}$ maps to the stack of semistable objects $A$ of $\cB_{p_\alpha}$ which have charge $(\alpha_{<k},1)$. Let $E$ be any semistable object of $\cB_{p_\alpha}$ of charge $(\alpha_k,0)$ (i.e. semistable sheaf on $S$ of charge $\alpha_k$). 

All extensions from $E$ to $F$ or from $F$ to $E$ giving semistable objects of $\cB_{p_\alpha}$ of charge $(\alpha_{\le k},1)$. Therefore \eqref{A*E} maps to the stack of semistable objects of $\cB_{p_\alpha}$ of charge $(\alpha_{\le k},1)$, and the induction continues. Moreover,
$$
ext^1_S(E,A)-hom_S(E,A)\=-\chi\_S(\alpha_k,\alpha_{<k})
$$
with no other Exts, and
$$
ext^1(A,E)-hom(A,E)\=\chi(\alpha_k(n)))-\chi\_S(\alpha_{<k},\alpha_k),
$$
also with no further Exts. Therefore by the scissor relations and
\eqref{extstack} (in constructible families) we have 
$$
[A^{(\alpha_{<k},1)},\epsilon_{\alpha_k}]\=\(\LL^{-\chi\_S(\alpha_k,\alpha_{<k})}-\LL^{\chi(\alpha_k(n)))-\chi\_S(\alpha_{<k},\alpha_k)}\)A^{(\alpha_{<k},1)}\times\epsilon_{\alpha_k}
$$
as a $\Q$-linear combination of stacks. Since all $\alpha_i=\delta_i\alpha$ are proportional to $\alpha$ we see that $\chi\_S(\alpha_{<k},\alpha_k)$ is symmetric in its arguments, and
$$
[A^{(\alpha_{<k},1)},\epsilon_{\alpha_k}]\=\LL^{-\chi\_S(\alpha_k,\alpha_{<k})}\(1-\LL^{\chi(\alpha_k(n))}\)A^{(\alpha_{<k},1)}\times\epsilon_{\alpha_k}.
$$
Therefore the Lie brackets in \eqref{eee} give in total
$$
\LL^{-\chi\_S(\alpha_2,\alpha_{<2})-\chi\_S(\alpha_3,\alpha_{<3})-\ldots-\chi\_S(\alpha_\ell,\alpha_{<\ell})}\left(\prod_{i=1}^\ell(1-\LL^{\chi(\alpha_i(n))})\;\epsilon_{\alpha_i}\right)\frac1{\LL-1}\,,
$$
where the final term comes from $\epsilon^{(0,1)}\cong B\C^*$.
Taking $(t-1)\chi\_{-t}$ of this gives, by \eqref{JDT},
$$
t^{-\sum_{i=2}^\ell\chi\_S(\alpha_i,\alpha_{<i})}\prod_{i-1}^\ell(-1)^{\chi\_S(\alpha_i,\alpha_i)}t^{\frac12(\chi(\alpha_i(n))-1)}[\chi(\alpha_i(n))]\_tt^{\frac12(1-\chi\_S(\alpha_i,\alpha_i))}\mathsf{J}_{\alpha_i}(t).
$$
By expressing $\chi\_S(\alpha,\alpha)$ as
$$
\sum_{i=1}^\ell\chi\_S(\alpha_i,\alpha_i)+2\sum_{i>j}\chi\_S(\alpha_i,\alpha_j)\=\sum_{i=1}^\ell\chi\_S(\alpha_i,\alpha_i)+2\sum_{i=2}^\ell\chi\_S(\alpha_i,\alpha_{<i})
$$
this can be rewritten
$$
(-1)^{\chi\_S(\alpha,\alpha)}t^{\frac12\chi(\alpha(n))-\frac12\chi\_S(\alpha,\alpha)}\prod_{i-1}^\ell[\chi(\alpha_i(n))]\_t\mathsf{J}_{\alpha_i}(t).
$$
Therefore $(t-1)\chi\_{-t}(\ \cdot\ )$ applied to \eqref{eee} gives
\begin{multline*}
(t-1)\chi\_{-t}(\cP^\perp_{\alpha,n}/\C^*)\=
(-1)^{\chi\_S(\alpha,\alpha)}t^{\frac12\dim\cP^\perp_{\alpha,n}} \\
\times\mathop{\sum_{\ell\ge 1,\,(\alpha_i=\delta_i\alpha)_{i=1}^\ell:}}_{\delta_i>0,\ \sum_{i=1}^\ell\delta_i=1}\frac{(-1)^\ell}{\ell!}
\prod_{i-1}^\ell[\chi(\alpha_i(n))]\_t\,\mathsf{J}_{\alpha_i}(t).
\end{multline*}
Multiplying both sides by 
$(-1)^{\dim\cP^\perp_{\alpha,n}}t^{-\frac12\dim\cP^\perp_{\alpha,n}}$
gives, by \eqref{refP},
$$
P^\perp_{\alpha,n}\=(-1)^{\chi(\alpha(n))}\mathop{\sum_{\ell\ge 1,\,(\alpha_i=\delta_i\alpha)_{i=1}^\ell:}}_{\delta_i>0,\ \sum_{i=1}^\ell\delta_i=1}\frac{(-1)^\ell}{\ell!}\,\prod_{i-1}^\ell[\chi(\alpha_i(n))]\_t\,\mathsf{J}_{\alpha_i}(t),
$$
which is the result claimed.
\end{proof}

\subsection{\for{toc}{$p_g(S)>0$ cosection and vanishing theorem}\except{toc}{${\bf p_g(S)>0}$ cosection and vanishing theorem.}}\label{sectoc}
Fix a surface $S$ with $p_g(S)>0$. In his calculations \cite{La1} Ties Laarakker observed a certain vanishing (more-or-less Corollary \ref{corsec} below, in low ranks). Here we explain it by describing a certain cosection of the fixed obstruction theory, on any connected component $\cP$ of the monopole branch of the $T$-fixed loci $(\cP^\perp_{\alpha,n})^T$. If stable$\,=\,$semistable the same construction can be done with moduli spaces $\cN\subset(\cN^\perp_\alpha)^T$ of sheaves on $X$ (or Higgs pairs on $S$) instead.

The construction is basically the same as in \cite[Section 5]{KKV}; we simply replace stable pairs by Joyce-Song pairs and impose the centre-of-mass-zero condition\footnote{Hence we get only one of the two cosections of \cite[Equation 5.8]{KKV}. This is sufficient since the Vafa-Witten $R\Hom(I\udot,I\udot)\_\perp[1]$ obstruction theory of \cite{TT1,TT2} has already had an $H^2(\cO_S)$ term removed from the obstruction sheaf, i.e. it is a \emph{reduced} obstruction theory in the language of \cite{KKV}.} on each $\C$ fibre of $X=S\times\C$ (equivalently, the $\tr\phi=0$ condition on the Higgs field).

We find it convenient to describe the cosection using Higgs data on $S$. For a more geometric description using sheaves on $X$ instead, see \cite[Section 5]{KKV}. By the symmetry of the obstruction theory, what we require is a $T$-weight one $\cP^\perp_{\alpha,n}$ vector field along (but not tangent to) $\cP$. This is Serre dual to a weight zero (i.e. invariant) cosection of the obstruction sheaf.

Since Joyce-Song stable pairs have no automorphisms and form a fine moduli space, there is a universal Higgs pair and Joyce-Song section $(\E,\Phi,\mathsf{s})$ over the total space of $p\_S\colon S\times\cP\to\cP$. Thus
$$
\Phi\in\Hom(\E,\E\otimes K_S\t) \qquad\text{and}\qquad \mathsf{s}\in H^0(\E(n)),
$$
where $n\gg0$ is sufficiently large that $\E(n)$ has no higher cohomology on any $S$ fibre.

Furthermore this lack of automorphisms means the universal $T$-fixed sheaf $\E$ admits a $T$-linearisation. (For any $\lambda\in T$ we get a \emph{unique} isomorphism $\phi_\lambda\colon(E,\phi,s)\to\lambda^*(E,\phi,s)=(E,\lambda\phi,s)$. Uniqueness then implies that $\phi_\lambda\circ\phi_\mu=\phi_{\lambda\mu}$.) Tensoring by its highest weight, we may assume that
\beq{wtspace}
\E\=\bigoplus_{i=0}^{r-1}\E_i\otimes\t^{-i}
\eeq
with each of $\E_0$ and $\E_{r-1}$ \emph{nonzero} and all $\E_i$ flat over $\cP$. Since we are on the monopole branch, $r\ge2$.
The weight one $\Phi$ maps each $\E_i$ to $\E_{i+1}$, and $\mathsf s$ is $T$-fixed up to automorphisms of $\E$.

Fix a nonzero holomorphic 2-form
$$
0\ \ne\ \sigma\ \in\ H^0(K_S)
$$
and consider the trace-free endomorphism
\beq{id}
\dot\varphi\ :=\ \left(\id_{\;\E_{r-1}}-\,\frac{\rk(\E_{r-1})}{\rk(\E)}\id_{\;\E}\right)\!\otimes\sigma\ \in\ \Hom(\E,\E\otimes K_S\)\_0\otimes\t.
\eeq
It defines a family of Higgs triples\footnote{This flow is most easily understood in terms of sheaves on $X$ \cite[Section 5]{KKV}; for instance it takes the sheaf $\cO_{rS}$ at $t=0$ to $\cO_{(r-1)\Gamma_{-t\sigma/r}}\oplus\cO_{\Gamma_{t\sigma}}$ at time $t\ne0$, where $\Gamma_\sigma\subset K_S$ is the graph of $\sigma$.}
\beq{flow}
(\E,\Phi+t\dot\varphi,\mathsf{s}) \quad\text{over}\quad \cP\times\C_t,
\eeq
where $t$ is the parameter on $\C=\C_t$. This family is classified by a map
$$
\cP\times\C_t\ \To\ \cP^\perp_{\alpha,n}.
$$
Differentiating at $t=0$ (equivalently, restricting from $\C_t$ to $\Spec\C[t]/(t^2)$) we get a $\cP^\perp_{\alpha,n}$ vector field on $\cP$,
\beq{phidot}
\dot\varphi\ \in\ \Gamma\Big(\cP,\,T_{\cP^\perp_{\alpha,n}}\big|_{\cP}	\Big).
\eeq
By construction it has $T$ weight one. The symmetry of the obstruction theory 
$$
\big(R\Hom(I\udot,I\udot)\_\perp[1]\big)^\vee\cong R\Hom(I\udot,I\udot)\_\perp\otimes\t^{-1}[2]
$$
makes $\Omega_{\cP}\cong\mathrm{Ob}_{\;\cP}\otimes\t^{-1}$, so \eqref{phidot} is equivalent to a weight zero cosection
\beq{cosec}
\mathrm{Ob}_{\;\cP}\ \To\ \cO_{\cP}.
\eeq
Its image is an ideal sheaf
$$
I_{Z(\dot\varphi)}\ \subseteq\ \cO_{\cP},
$$ 
so \eqref{cosec} has a well-defined zero scheme $Z(\dot\varphi)$.

\begin{thm} \label{cosecinj} Pick $0\ne\sigma\in H^{2,0}(S)$ and consider the cosection \eqref{cosec} on a connected component $\cP\subset(\cP^\perp_{\alpha,n})^T$ of the monopole branch.

At any closed point of the zero scheme $Z(\dot\varphi)\subset\cP$ the maps $\phi\colon E_i\to E_{i+1}$ are both injective and generically surjective\footnote{That is, they have torsion cokernel.} on $S$ for each $i=0,\ldots,r-2$. In particular if $Z(\dot\varphi)\neq\emptyset$ then $\rk \E_0=\rk\E_1=\cdots=\rk\E_{r-1}$.
\end{thm}

\begin{proof} First we discuss how basechange works in this setting. If $p=(E,\phi,s)$ is a point of $\cP$ with ideal $\mathfrak m\subset\cO_{\cP}$ then $\Omega_{\cP}|_p\cong\mathfrak m/\mathfrak m^2$ and \eqref{cosec} induces a map $\mathfrak m/\mathfrak m^2\to\cO_p\cong\C$ which is zero if and only if $p\in Z(\dot\varphi)$. Therefore $p\in Z(\dot\varphi)$ if and only if the vector field \eqref{phidot} restricted to $p$ maps to zero under the natural map (which need not be an isomorphism!)
\beq{tanmap}
T_{\cP}\big|_p\ \To\ \(\mathfrak m/\mathfrak m^2\)^*.
\eeq
But $\(\mathfrak m/\mathfrak m^2\)^*$ is described by deformation theory \cite{TT2} as $\Ext^1_X(I\udot,I\udot)\_\perp$. By forgetting the section $s$, this maps (see \cite[Equation 6.11]{TT2}, for instance)
\beq{tanmap2}
\Ext^1_X(I\udot,I\udot)\_\perp\ \To\ \Ext^1_X(\cE_\phi,\cE_\phi)\_\perp
\eeq
to the first order deformation space of $(E,\phi)$, which sits in the exact sequence
\beq{amicseq}
0\To\Hom_S(E,E)\_0\rt{[\ \cdot\ ,\phi]}\Hom_S(E,E\otimes K_S)\_0\;\t\To\Ext^1_X(\cE_\phi,\cE_\phi)\_\perp
\eeq
of \cite[Equations 2.20 and 5.32]{TT1}. By \eqref{id} the image of the vector field \eqref{phidot} under first \eqref{tanmap} and then \eqref{tanmap2} can be seen in this exact sequence as the image of 
\beq{id2}
\left(\id_{E_{r-1}}-\,\frac{\rk(E_{r-1})}{\rk(E)}\id_{E}\right)\!\otimes\sigma\ \in\ \Hom_S(E,E\otimes K_S\)\_0
\eeq
in the second term. So to prove that $p\not\in Z(\dot\varphi)$ it is sufficient to show that \eqref{id2} is not in the image of the first term of the sequence \eqref{amicseq}.\medskip

Assume first that at the point $p=(E=\oplus_{i=0}^{r-1}E_i\t^{-i},\phi,s)\in\cP$, one of the $\phi\colon E_i\to E_{i+1}$ fails to be injective on $S$. Letting $K_i$ denote the kernel of
\beq{phipower}
\phi^{r-1-i}\big|_{E_i}\,\colon\ E_i\ \To\ E_{r-1},
\eeq
our assumption is that
$$
K\ :=\ \bigoplus_{i=0}^{r-2}K_i\ \ne\ 0.
$$
Since $(E,\phi)$ is semistable it is torsion free, so all of the $K_i\subset E_i$ are torsion free. There is an open subset $U\subset S$ over which they are all locally free, $\phi$ has constant rank, and --- we claim --- there is a $\phi$-invariant splitting
\beq{split}
E|_U\ \cong\ K|_U\oplus(E/K)|_U.
\eeq
To prove the claim, we may shrink $U$ if necessary and then split $K_0\subset E_0$ over it. Then we proceed inductively. At the $i$th stage we have split $E_i=K_i\oplus E_i/K_i$ over $U$, which we use to form $\phi(E_i/K_i)\subset E_{i+1}$. Since its intersection with $K_{i+1}\subset E_{i+1}$ is 0 we have $\phi(E_i/K_i)\oplus K_{i+1}\subset E_{i+1}$. Further shrinking $U$ if necessary we can find a complement $C_i$ to this subbundle. Then $C_i\oplus\phi(E_i/K_i)$ gives the required complement to $K_{i+1}\subset E_{i+1}$.

Therefore \eqref{split} defines a local splitting of Higgs bundles over $U$. The endomorphism $\dot\varphi|_U$ of \eqref{id2} acts on the first summand as
$$
-\frac{\rk(E_{r-1})}{\rk(E)}\id_K\!\otimes\;\sigma \quad\text{of trace }
-\frac{\rk(E_{r-1})\rk(K)}{\rk(E)}\;\sigma\ \in\,H^0(K_S|_U).
$$
Therefore it is not in the image of the map $[\ \cdot\ ,\phi]$ in the sequence \eqref{amicseq}, since commutators are trace-free. That is, it defines a nonzero deformation of the Higgs pair $(K,\phi|_K)$ over $U$. Therefore the cosection \eqref{cosec} is nonzero at $p$.\medskip

So now we turn to the case where $K=0$ but at least one of the maps \eqref{phipower} has cokernel of rank $>0$. Then there is an open set $U\subset S$ over which the cokernels are locally free, $\rk(\phi)$ is constant, and --- we claim --- there is a $\phi$-invariant splitting of locally free sheaves
$$
E|_U\ \cong\ \bigoplus_{i=0}^{k-1}(\phi^iE_0|_U\oplus C_i)
$$
with $C_{k-1}\ne0$. We prove the claim inductively, with base case $C_0:=0$. At the $i$th stage, possibly after shrinking $U$ as usual, we pick a complement $D_{i+1}$ to $\phi^{i+1}(E_0)$ and then set $C_{i+1}=D_{i+1}\oplus\phi(C_i)$.

Therefore $\bigoplus_{i=0}^{k-1}\phi^iE_0|_U$ is a proper sub Higgs bundle of $(E,\phi)|_U$. On restriction to it, the endomorphism \eqref{id2} acts as
$$
\left(\id_{\phi^{k-1}(E_0)}-\,\frac{\rk(E_{k-1})}{\rk(E)}\id\right)\otimes\;\sigma
$$
whose trace is $<\rk(\phi^{k-1}(E_0))-\rk(E_{k-1})=-\rk(C_{k-1})<0$. Therefore it is not in the image of the map $[\ \cdot\ ,\phi]$ in the sequence \eqref{amicseq}. That is, it defines a nonzero deformation of the sub Higgs bundle over $U$, and again the cosection \eqref{cosec} nonzero at $p=(E,\phi,s)$.
\end{proof}

These results can be strengthened, for instance by showing that at points of $Z(\dot\varphi)$ the cokernels of $\phi\colon E_i\to E_{i+1}$  must have support on the zero divisor of $\sigma$ --- see Theorem \ref{K3cosec} below for K3 surfaces, for example. But since 
a nowhere zero cosection on $\cP$ forces its contribution to the $K$-theoretic invariant to vanish (by an easy special case of \cite{KL}; see for instance \cite[Proposition 3.2]{KL}) the above result is enough to prove that \emph{most} components $\cP$ do not contribute to the invariants. In particular the most obvious special case is the following.

\begin{cor} \label{corsec}
If $h^{2,0}(S)>0$ and the multi-rank is non-constant,
$$
(\rk E_0,\rk E_1,\ldots,\rk E_{r-1})\ \ne\ (a,a,\ldots,a) \quad\forall a\in\N
$$
then $\cP$ does not contribute to the refined invariants.
\end{cor}

\begin{rmk}\label{urmk}
The cosection does not rule out the contribution of \emph{all} components. For instance, when $\Phi\colon\E_i\to\E_{i+1}$ is an isomorphism for $i=0,\ldots,r-2$, the cosection \eqref{cosec} vanishes identically. This is because the endomorphism \eqref{id} can be written
$$
\dot\varphi\=[A,\Phi],
$$
where $A\colon\E\to\E$ acts as zero on $\E_0$ and as $\Phi^{-1}\colon\E_i\to\E_{i-1}$ on any other summand. Therefore, in the exact sequence
\beq{tmap}
0\To\hom_{p\_S}(\E,\E)\_0\rt{[\ \cdot\ ,\Phi]}\hom_{p\_S}(\E,\E\otimes K_S)\_0\;\t\To\ext^1_{p\_X}\!(\curly E_\Phi,\curly E_\Phi)\_\perp
\eeq
of \cite[Equations 2.20 and 5.32]{TT1}, we see that $\dot\varphi$ maps to zero in the deformation space of $\curly E_\Phi$ on $p\_X\colon X\times\cP\to\cP$. That is, the first order deformation of $\curly E_\Phi$ (or equivalently $(\E,\Phi)$) is zero. Since $\mathsf s$ is constant in the flow \eqref{flow}, we see the tangent vector \eqref{phidot} vanishes identically so $Z(\dot\varphi)=\cP$.
\end{rmk}

\subsection{K3 surfaces}\label{kthree}
The cosection \eqref{cosec} gives the strongest results on (polarised) $K3$ surfaces $(S,\cO_S(1))$. Just as for the stable pairs in \cite{KKV}, it will show that the only Joyce-Song pairs on $S\times\C$ which contribute nontrivially to the refined Vafa-Witten invariants are those which have \emph{constant thickening} in the $\C$-direction. That is, they are pulled back from $S$ then tensored by $\cO_{rS}:=\cO_X/I_{S\subset X}^r$ for some $r>0$:
$$
\cE\ =\ \rho^*E_0\otimes\cO_{rS}.
$$
In Higgs language these correspond to the following.

\begin{defn} \label{unif} We call a point $(E,\phi,s)\in(\cP^\perp_{\alpha,n})^T$ \emph{uniform} if the maps \eqref{phipower} are all isomorphisms. Equivalently,
$$
E\ \cong\ \bigoplus_{i=0}^{k-1}\phi^i(E_0)\;\t^{-i}.
$$
\end{defn}

\begin{thm} \label{K3cosec} A connected component $\cP\subset(\cP^\perp_{\alpha,n})^T$ contains a uniform point if and only if all its points are uniform, if and only if the cosection \eqref{cosec} vanishes identically on $\cP$. 

Otherwise the cosection is nowhere zero on $\cP$.
In particular, non-uniform components $\cP$ contribute zero to the refined invariants.
\end{thm}

\begin{proof}
Firstly we claim that for any semistable Higgs pair $(E,\phi)$ on $S$ the underlying sheaf $E$ is also semistable.

Let $F\subset E$ be the first term of its Harder-Narasimhan filtration. Therefore $F$ is semistable with strictly larger reduced Hilbert polynomial (or ``Gieseker slope") than the other graded pieces of the filtration. Since those pieces are also semistable, it follows that
$$
\Hom(F,E/F)\=0.
$$
Therefore the Higgs field $\phi$ preserves $F$, so $F$ either strictly destabilises $(E,\phi)$ or is all of $E$. \medskip

In particular each $E_i$ in the weight space decomposition of $E$ is semistable (and so torsion-free) of the same reduced Hilbert polynomial. Therefore if $\phi\colon E_i\to E_{i+1}$ is injective and generically surjective, it is an isomorphism.

It therefore follows from Theorem \ref{cosecinj} that any closed point of the zero locus $Z(\dot\varphi)$ of the cosection is uniform in the sense of Definition \ref{unif}. Since being uniform is an open condition, while the zero locus $Z(\dot\varphi)$ is closed, a single uniform point makes the whole connected component $\cP$ uniform.
In this case the cosection vanishes identically on $\cP$ by Remark \ref{urmk}.
\end{proof}

\subsection{Refined multiple cover formula} When $\cO_S(1)$ is generic for charge $\alpha$ this leaves only the uniform components to calculate on:
\beq{unifor}
P^\perp_{\alpha,n}\=\sum_{r|\alpha}\int_{\Big[\cP^\perp_{\left(\frac\alpha r\right)^{\!\;r}\!,\,n}\Big]^{\vir}}\frac1{e(N^{\vir})}\,.
\eeq
Here $\cP^\perp_{\left(\frac\alpha r\right)^{\!\;r}\!,\,n}$ is the moduli space of uniform Joyce-Song pairs of charge $\alpha$ which are $r$-times thickened pairs of charge $\alpha/r$. They are determined by their restriction to $S$, giving an isomorphism
\beqa
\cP^\perp_{\left(\frac\alpha r\right)^{\!\;r}\!,\,n} &\cong& \cP^S_{\alpha/r,n}\,, \\
(\rho^*E\otimes\cO_{rS},\rho^*s) &{\ensuremath{\leftarrow\joinrel\relbar\joinrel\shortmid}}& (E,s).
\eeqa
Here $\cP^S_{\alpha/r,n}$ denotes the moduli space of Joyce-Song stable pairs $(E,s)$ on $(S,\cO_S(1))$ with charge $\alpha/r$: so $E$ is a Gieseker semistable sheaf on $S$ with total Chern class $\alpha/r$, and $s\in H^0(E(n))$ does not factor through any destabilising subsheaf of $E$.
%
We noted in Theorem \ref{K3cosec} this gives a bijection of sets. That this makes $\cP^S_{\alpha/r,n}$ \emph{scheme theoretically} isomorphic to a component of $(\cP^\perp_{\alpha,n})^T$ follows from the the deformation theory analysis \eqref{fixt} below, or by \cite[Lemma 1]{KKV}. \medskip

Fix $r$ and let $\alpha_0=\alpha/r$. Both $\cP^S:=\cP^S_{\alpha_0,n}$ and $(\cP^\perp_{r\alpha_0,n})^T$ are fine moduli spaces. We denote the two universal sheaves by
$$
\E_0\,\text{ on }\,S\times\cP^S \qquad\text{and}\qquad \curly E=\rho^*\E_0\otimes\cO_{rS\times\cP^S}\,\text{ on }\,X\times\cP^S,
$$
and the universal complexes made from the Joyce-Song pairs by
$$
I\udot_S\ :=\ \{\cO_S(-n)\To\E_0\} \qquad\text{and}\qquad I\udot_X\ :=\ \{\cO_X(-n)\To\curly E\},
$$
with $\cO(-n)$ in degree 0. As usual $p\_S\colon S\times\cP^S\to\cP^S$ and $p\_X\colon X\times\cP^S\to\cP^S$ denote the projections. Now let
$$
E\udot_S\ :=\ R\hom_{p\_S}(I\udot_S,\E_0)^\vee.
$$
Though we will not need to know it, this is the virtual cotangent bundle of the natural perfect obstruction theory on the moduli space $\cP^S$ of pairs on $S$; see \cite{KT1} for instance.

When considered instead as a moduli space of sheaves on $X$ (not yet imposing the centre-of-mass-zero condition), we get a the virtual cotangent bundle 
$$
R\hom_{p\_{X\!}}(I\udot_X,I\udot_X)_0^\vee[-1]\ \cong\ 
E\udot_S\otimes(\t^0\oplus\cdots\oplus\t^{r-1})\ \oplus\ (E\udot_S)^\vee[1]\otimes(\t^{-1}\oplus\cdots\oplus\t^{-r})
$$
by \cite[Proposition 4]{KKV}. (While that paper works with stable pairs, it uses none of their special properties --- the proof goes through verbatim for Joyce-Song pairs.)

To get the Vafa-Witten obstruction theory we remove $H^0(K_S)\otimes\t$ from the deformations (imposing the centre-of-mass-zero condition) and --- dually --- $H^2(\cO_S)$ from the obstructions, replacing $E_S\udot$ by the \emph{reduced} obstruction theory of \cite{KT1} (we do not prove this compatibility of obstruction theories since we do not need it; we only require the $K$-theory class of the virtual (co)tangent bundle).
The upshot is that the Vafa-Witten obstruction theory of \cite{TT2} is, on restriction to $\cP^S\subset\cP^\perp_{r\alpha_0,n\,}$, the $K$-theory class of
\begin{multline*}
\LL^{\vir}_{\cP^\perp_{r\alpha,n\!}}\Big|_{\cP^S}\ \cong\ 
\LL^{\red}_{\cP^S}\otimes(\t^0\oplus\cdots\oplus\t^{r-1})-(\t\oplus\cdots\oplus\t^{r-1})\\
-(\LL^{\red}_{\cP^S})^\vee\otimes(\t^{-1}\oplus\cdots\oplus\t^{-r})+(\t^{-2}\oplus\cdots\oplus\t^{-r}).
\end{multline*}
If we write
$$
\Psi\ :=\ \LL^{\red}_{\cP^S}\otimes(\t\oplus\cdots\oplus\t^{r-1})-(\t^2\oplus\cdots\oplus\t^{r-1})
$$
then
$$
\LL^{\vir}_{\cP^\perp_{r\alpha,n\!}}\Big|_{\cP^S}\ \cong\ \LL_S^{\red}+\Psi-\t-\Psi^\vee-(\LL_S^{\red})^\vee\t^{-r}+\t^{-r}.
$$
From this we read off
$$
K^{\frac12}_{\vir}\big|_{\cP^S}\ \cong\ \det(\Psi)\t^{-(r+1)/2}\det(\LL_S^{\red})\t^{r\vd/2},
$$
where $\vd:=\rk(\LL^{\red}_S)$ is the reduced virtual dimension of $\cP^S$,
\beq{fixt}
\Big(\LL^{\vir}_{\cP^\perp_{r\alpha,n\!}}\Big|_{\cP^S}\Big)^{\mathrm{fix}}\ \cong\ \LL_S^{\red},
\eeq
and
$$
(N^{\vir})^\vee\ \cong\ \Psi-\Psi^\vee-(\LL_S^{\red})^\vee\t^{-r}+\t^{-r}-\t.
$$
Therefore
\begin{multline*}
\frac{\cO^{\vir}_{\cP^S}\otimes K_{\vir}^{\frac12}\big|_{\cP^S}}{\Lambda\udot(N^{\vir})^\vee} \\
\=\frac{\det(\Psi)\otimes\Lambda\udot(\Psi^\vee)}{\Lambda\udot(\Psi)}
\Lambda\udot(\LL_S^{\red}\t^r)^\vee\!\det(\LL_S^{\red}\t^r)\t^{-\frac12(r\vd+r+1)}\frac{1-\t}{1-\t^{-r}}\cO^{\vir}_{\cP^S} \\
\=(-1)^{\rk(\Psi)}
(-1)^{\vd}\t^{-\frac12r\vd}\Lambda\udot(\LL_S^{\red}\t^r)\frac{(-1)}{[r]\_t}\cO^{\vir}_{\cP^S}
\end{multline*}
by two applications of the identity $\det(E)^*\otimes\Lambda\udot E\big/\Lambda\udot E^\vee=(-1)^{\rk(E)}$ of \eqref{indenti}.

Substituting $\rk(\Psi)=(r-1)\vd-(r-2)$ and $\vd\equiv\chi(\alpha_0(n))-1\pmod2$ and taking $\chi\_t$ gives the following contribution to the refined pairs invariants,
\beq{ansa1}
P^\perp_{(\alpha\_0)^r,n}\=\frac{(-1)^{\chi(r\alpha\_0(n))-1}}{[r]\_t}t^{-\frac12r\vd}\chi^{\vir}_{-t^r}(\cP^S_{\alpha\_0,n}),
\eeq
cf. \eqref{shifted}.
This gives a refined multiple cover formula under any of the conditions
\begin{enumerate}
\item $\alpha_0=\alpha/r$ is primitive, or
\item all semistable sheaves of class $\alpha_0$ on $S$ are stable, or
\item Conjecture \ref{pechconj} holds for $\alpha_0$.
\end{enumerate}
Notice that $(1)\so(2)$ while Proposition \ref{s-ss} shows $(2)\so(3)$.

\begin{thm}\label{ball}
Suppose any of (1),\,(2) or (3) holds for the charge $\alpha_0$.
Then the contribution \eqref{ansa1} of uniform Joyce-Song pairs satisfies Conjecture \ref{pechconj} for all $r$. The resulting refined Vafa-Witten invariants are given by the multiple cover formula
\beq{refmc}
\VW_{(\alpha_0)^r}(t)\=\frac{\VW_{(\alpha\_0)^1}(t^r)}{[r]^2_t}\,.
\eeq
Furthermore, when $\alpha_0$ is primitive, setting $d:=1-\frac12\chi\_S(\alpha_0,\alpha_0)$ we have
$$
\VW_{(\alpha\_0)^r}(t)\=\frac{t^{-rd}\chi\_{-t^r}(\Hilb^dS)}{[r]_t^2}\,.
$$
\end{thm}

\begin{proof}
Since we are assuming Conjecture \ref{pechconj} holds for $\alpha_0$, we may apply \eqref{desk2} to the $r=1$ instance of \eqref{ansa1} to deduce
\beq{1case}
\VW_{(\alpha\_0)^1}(t)\=t^{-\frac12\vd}\frac{\chi^{\vir}_{-t}(\cP^S)}{[\chi(\alpha_0(n))]\_t}\,.
\eeq
Combining this with the identity $[\chi]\_{t^r}[r]\_t=[r\chi]\_t$, we rewrite \eqref{ansa1} as
\beq{ansa}
(-1)^{\chi(r\alpha\_0(n))-1}[\chi(r\alpha_0(n))]\_{t\,}\frac{\VW_{(\alpha\_0)^1}(t^r)}{[r]^2_t}\,.
\eeq
Therefore \eqref{desk2} holds for all $k>1$ with
\[
\VW_{(\alpha\_0)^r}(t)\=\frac{\VW_{(\alpha\_0)^1}(t^r)}{[r]^2_t}\,.
\]
When $\alpha_0$ is primitive, $\cP^S=\cP^S_{\alpha\_0,n}$ is a $\PP^{\chi(\alpha_0(n))-1}$-bundle over the moduli space $\cM_S$ of (semistable$\,=\,$stable) sheaves on $S$ of class $\alpha_0$. In turn, $\cM_S$ is smooth with $\LL^{\red}_S\cong\Omega_{\cM_S}$ and deformation equivalent to $\Hilb^dS$, where $d=1-\frac12\chi\_S(\alpha_0,\alpha_0)$. So pushing $\Lambda\udot(\Omega_{\cP^S}\t)$ down to $\cM_S$ and then taking $\chi\_t$ gives, just as in \eqref{quant},
$$
\chi\_{-t}(\cP^S)\=(1+t+\cdots+t^{\chi(\alpha\_0(n))-1})\chi\_{-t}(\Hilb^dS).
$$
Substituting this into \eqref{1case} gives
\begin{align*}
\VW_{(\alpha\_0)^1}(t) &=\ t^{-\frac12\(2d+\chi(\alpha\_0(n))-1\)}t^{\frac12(\chi(\alpha\_0(n))-1)}[\chi(\alpha_0(n))]\_t\frac{\chi\_{-t}(\Hilb^dS)}{[\chi(\alpha_0(n))]\_t} \\ &=\ t^{-d}\chi\_{-t}(\Hilb^dS), \qquad d=1-\frac12\chi\_S(\alpha_0,\alpha_0).\qedhere
\end{align*}
\end{proof}

So \eqref{refmc} turns out to be the correct refinement of the more familiar multiple cover formula
$$
\VW_{(\alpha\_0)^r}\=\frac1{r^2}\VW_{(\alpha\_0)^1}.
$$
For these numerical invariants it is known (by \cite[Theorem 6.25]{TT2}, which builds on results of \cite{MT1, KKV, To}) that
$$
\VW_{(\alpha\_0)^1}\=e(\Hilb^dS), \qquad d=1-\frac12\chi\_S(\alpha_0,\alpha_0),
$$
even when $\alpha_0$ is not primitive. That is, the contribution of $T$-fixed semistable sheaves \emph{scheme theoretically supported on $S$} to the Vafa-Witten invariants is what we get for primitive $\alpha_0$.

It seems natural to conjecture the same for the refined invariants.
Summing over all uniform multiple covers using \eqref{unifor},
\beq{NiS}
P^\perp_{\alpha,n}\=(-1)^{\chi(\alpha(n))-1}\sum_{r|\alpha}
\frac{t^{-\frac12r\vd}\chi^{\vir}_{-t^r}\(\cP^S_{\frac\alpha r,n}\)}{[r]\_t}\,,
\eeq
then substituting in \eqref{refmc} gives the following.

\begin{conj}\label{K3conj}
If $\cO_S(1)$ is generic in the sense of \cite[Equation 2.4]{TT2} then
\beq{newno}
\VW_{\alpha}(t)\=\sum_{r|\alpha}\frac{t^{\chi\_S(\alpha,\alpha)/2r-r}\chi\_{-t^r}(\Hilb^{1-\chi\_S(\alpha,\alpha)/2r^2\!}S)}{[r]_t^2}\,.
\eeq
\end{conj}

When $\alpha=r\alpha_0$ with $r$ prime and $\alpha_0$ primitive, this becomes
\begin{multline*}
\VW_{r\alpha}(t)\=t^{r^2\chi\_S(\alpha,\alpha)/2-1}\chi\_{-t}(\Hilb^{1-r^2\chi\_S(\alpha,\alpha)/2\!}S) \\ +\frac{t^{r(\chi\_S(\alpha,\alpha)/2-1)}\chi\_{-t^r}(\Hilb^{1-\chi\_S(\alpha,\alpha)/2\!}S)}{[r]_t^2}
\end{multline*}
and was already conjectured by G\"ottsche-Kool \cite{GK3}. In the published version of this paper we claimed to prove this G\"ottsche-Kool Conjecture by combining \eqref{NiS} and Theorem \ref{ball}, but Martijn Kool pointed this only handles the last term above. However, the more general Conjecture \ref{K3conj} has now been proved in \cite{KK3}.
 
\begin{thm}\label{K3proof} Conjecture \ref{K3conj} is true \cite{KK3}.$\hfill\square$
\end{thm}


Taking $\alpha=(r,0,k)$ in \eqref{newno} gives
$$
\sum_k\VW_{r,k}(t)q^k\ =\ \sum_{d|r}\frac1{[d]_t^2}\sum_{m\in\Z}
t^{-r(m-r/d)-d}\chi\_{-t^d}\Big(\Hilb^{\frac rd\big(m-\frac rd\big)+1}S\Big)q^{md},
$$
where on the right we have summed over those second Chern classes $k=md$ divisible by $d$. Shifting $m$ by the integer $r/d$ simplifies this to
\beq{temp}
\sum_{d|r}\frac1{[d]_t^2}\sum_{m\in\Z}
t^{-mr-d}\chi\_{-t^d}\Big(\!\Hilb^{\frac rdm+1}S\Big)q^{md+r}.
\eeq
To sum this we write
$$
\sum_{k=-1}^\infty t^{-k-1}\chi\_{-t}(\Hilb^{k+1}S)\;q^k\ =\ \widetilde\Delta(q,t)^{-1}.
$$
By the results of G\"ottsche-Soergel \cite{GS}, it is the unique Jacobi cusp form of index 10 and weight 1,
\beq{deltadef}
\widetilde\Delta(q,t)\ :=\ q\prod_{k=1}^\infty(1-q^k)^{20}(1-tq^k)^2(1-t^{-1}q^k)^2.
\eeq
It specialises at $t=1$ to the modular form $\eta(q)^{24}$.

Taking only powers $k=\frac rdm$ of $q$ divisible by the integer $r/d$ on both sides of this formula, and substituting $t^d$ for $t$, gives
$$
\sum_mt^{-mr-d}\chi\_{-t^d}(\Hilb^{\frac rdm+1}S)q^{\frac rdm}\=\frac dr\sum_{j=0}^{r/d-1}\widetilde\Delta(e^{2dj\pi i/r}q,t^d)^{-1}.
$$
Substituting in \eqref{temp} gives

\begin{cor} If $\cO_S(1)$ is generic then
$$
\sum_n\VW_{r,n}(t)\;q^n\=\sum_{d|r}\frac1{[d]_t^2}\frac dr\,q^r\!\sum_{j=0}^{r/d-1}\widetilde\Delta\(e^{2dj\pi i/r}q^{\frac{d^2}r},t^d\)^{-1},
$$
where $\widetilde\Delta$ is given by \eqref{deltadef}. 
\end{cor}

When $r$ is prime this simplifies to
$$
\sum_n\VW_{r,n}(t)\;q^n\=\frac1{[r]_t^2}\,q^r\,\widetilde\Delta\(q^r,t^r\)^{-1}+\frac1rq^r\sum_{j=0}^{r-1}\widetilde\Delta\(e^{2j\pi i/r}q^{\frac1r},t\)^{-1},
$$
proving a conjecture of G\"ottsche-Kool \cite{GK3}.

\subsection{\for{toc}{$K_S>0$}\except{toc}{\bf $\mathbf{K_S>0}$}}
Finally we describe the refinement of a simple calculation on general type surfaces $S$ from \cite[Section 6.3]{TT2}. For more recent and much more general results we refer to the discussion of Laarakker's results \cite{La2} in the Introduction.

Take a surface $S$ with $h^{0,1}(S)=0$ and $h^{0,2}(S)>0$ and charge $\alpha=(2,0,0)\in H^*(S)$. The $\C^*$-fixed semistable sheaves on $X=K_S$ are
\begin{enumerate}
\item[(a)] (the pushforward from $S$ of) $\cO_S^{\oplus2}$, and
\item[(b)] (the pushforward from $2S\subset X$ of) $I_{C\subset 2S}\otimes K_S$, for any $C\subset S\subset2S\subset X$ in the canonical linear system $|K_S|$.
\end{enumerate}

Taking sections twisted by $\cO(n)$ for $n\gg0$ (using in (b) the fact that  the pushdown to $S$ of $I_{C\subset 2S}\otimes K_S$ is $\cO_S\oplus\cO_S\t^{-1}$), imposing stability and dividing by the automorphism group of the sheaf, we find the following. The moduli space of $\C^*$-fixed stable Joyce-Song pairs has two components,
\begin{enumerate}
\item[(a)] $\Gr(2,\Gamma(\cO_S(n)))$ of pairs with underlying sheaf $\cO_S^{\oplus2}$, and
\item[(b)] $\PP(\Gamma(\cO_S(n)))\times\PP(\Gamma(K_S))$ of pairs with underlying sheaf $I_{C\subset 2S}K_S$.
\end{enumerate}
The first component has a trivial $H^2(\cO_S)$ piece in the obstruction bundle $\ext^2_{p\_X}(I\udot,I\udot)\_\perp$, so it contributes nothing to the invariants. The second component is shown in \cite[Section 6.3]{TT2} to have fixed obstruction bundle
\beq{obbdl}
\mathrm{Ob}\ \cong\ T^*_{\PP(\Gamma(K_S))}
\eeq
pulled back from $\PP(\Gamma(K_S))$,
and virtual normal bundle
\begin{multline}\label{envir}
N^{\vir}\=
T_{\PP(\Gamma(\cO_S(n)))}\t^{-1}\ \oplus\ 
T_{\PP(\Gamma(K_S))}\t\ \oplus\ \Gamma(K_S)\!\otimes\!\cO_{\PP(\Gamma(K_S))}(1)\t^2 \\
-\ T^*_{\PP(\Gamma(\cO_S(n)))\times\PP(H^0(K_S))}\t\ -\ 
T^*_{\PP(\Gamma(\cO_S(n)))}\t^2\ -\ \Gamma(K_S)^*\otimes\cO_{\PP(\Gamma(K_S))}(1)\t^{-1}.
\end{multline}
A generic element of $\End_0\Gamma(K_S)\cong H^0(T_{\PP(\Gamma(K_S))})$ with distinct eigenvalues gives a vector field on $\PP(\Gamma(K_S))$ with $p_g(S)$ distinct zeros, and so a vector field on $\PP(\Gamma(\cO_S(n)))\times\PP(\Gamma(K_S))$ whose zero locus is $p_g(S)$ distinct $\PP(\Gamma(\cO_S(n)))$ fibres. By \eqref{obbdl} the corresponding Koszul resolution gives an equality in $K$-theory
$$
\cO^{\vir}\=\Lambda\udot\;\mathrm{Ob}^*\=(-1)^{p_g(S)-1}p_g(S)\cO_{\PP(\Gamma(\cO_S(n)))},
$$
so that the refined pairs invariant is
\beq{compoot}
P^\perp_\alpha(n,t)\=(-1)^{p_g(S)-1}p_g(S)\cdot\chi\_t\!\left(\PP(\Gamma(\cO_S(n))),\frac{K_{\vir}^{1/2}}{\Lambda\udot(N^{\vir})^\vee}\right).
\eeq
On a $\PP(\Gamma(\cO_S(n)))$ fibre the virtual normal bundle \eqref{envir} simplifies to
$$
N^{\vir}\=\Psi-\Psi^\vee
\ \oplus\ 
(\t^2)^{\oplus p_g(S)} \\
-\ T^*_{\PP(\Gamma(\cO_S(n)))}\t^2\ -\ (\t^{-1})^{\oplus p_g(S)},
$$
where $\Psi:=T_{\PP(\Gamma(\cO_S(n)))}\t^{-1}$.
Combined with \eqref{obbdl} this means $K_{\vir}$ is
$$
K_{\PP(\Gamma(\cO_S(n)))}\otimes(\det\Psi)^{-2}\t^{-2p_g(S)}\otimes K_{\PP(\Gamma(\cO_S(n)))}\t^{2\chi(\cO_S(n))-2}\t^{-p_g(S)}
$$
with square root
$$
K_{\vir}^{\frac12}\=K_{\PP(\Gamma(\cO_S(n)))}\otimes(\det\Psi)^{-1}\t^{-3p_g(S)/2}\t^{\chi(\cO_S(n))-1}.
$$
So we now calculate \eqref{compoot} to be
\begin{multline*}
(-1)^{p_g(S)-1}p_g(S)t^{\chi(\cO_S(n))-1-3p_g(S)/2}\frac{(1-t)^{p_g(S)}}{(1-t^{-2})^{p_g(S)}}\times \\
\chi\_t\!\left(\frac{\Lambda\udot\Psi\otimes
\Lambda\udot\;(T_{\PP(\cO_S(n))}\t^{-2})K_{\PP(\Gamma(\cO_S(n)))}}{\det\Psi\otimes\Lambda\udot\Psi^\vee}\right)\!.
\end{multline*}
By two applications of the identity \eqref{indenti} (and recalling that $\rk\Psi=\chi(\cO_S(n))-1$) this gives
\begin{multline*}
-p_g(S)\frac{t^{\chi(\cO_S(n))-1+p_g(S)/2}}{(1+t)^{p_g(S)}}\chi\_t
(\Lambda\udot\;(\Omega_{\PP(\Gamma(\cO_S(n)))}\t^2)t^{-2\chi(\cO_S(n))+2}
\\ \=-\frac{p_g(S)}{[2]_t^{p_g(S)}}\big[\chi(\cO_S(n))\big]_{t^2}\,.
\end{multline*}
Since $[\chi]\_{t^2}=[2\chi]\_t/[2]\_t$ this gives
$$
P^\perp_\alpha(n,t)\=-p_g(S)\frac{[2\chi(\cO_S(n)]\_t}{[2]_t^{p_g(S)+1}}\,.
$$
This fits Conjecture \eqref{pechconj} perfectly and makes the refined Vafa-Witten invariant 
\beq{end}
\VW_\alpha(t)\ =\ \frac{p_g(S)}{[2]_t^{p_g(S)+1}}\,.
\eeq

\bibliographystyle{halphanum}
\bibliography{references}

\bigskip \noindent {\tt{richard.thomas@imperial.ac.uk}} \medskip

\noindent Department of Mathematics \\
\noindent Imperial College London\\
\noindent London SW7 2AZ \\
\noindent United Kingdom

\end{document}